\documentclass[12pt,a4paper,reqno]{amsart}
\usepackage[noadjust]{cite}
\usepackage{amsfonts}
\usepackage{amsthm}
\usepackage{amsmath}
\usepackage{amscd}
\usepackage{comment}
\usepackage[applemac]{inputenc}
\usepackage{t1enc}
\usepackage[mathscr]{eucal}
\usepackage{indentfirst}
\usepackage[dvipsnames]{xcolor}
\usepackage{graphicx}
\usepackage{graphics}
\usepackage{pict2e}
\usepackage{epic}
\numberwithin{equation}{section}
\usepackage[section]{placeins}
\usepackage[margin=2.9cm]{geometry}
\usepackage{epstopdf}
\usepackage{hyperref} 

\theoremstyle{plain}
\newtheorem{Th}{Theorem}[section]
\newtheorem{Lemma}[Th]{Lemma}
\newtheorem{Corollary}[Th]{Corollary}
\newtheorem{Proposition}[Th]{Proposition}

\theoremstyle{definition}
\newtheorem{Definition}[Th]{Definition}

\newtheorem{Remark}[Th]{Remark}
\newtheorem{?}[Th]{Problem}

\newcommand{\barC}{{\overline{C}}}

\begin{document}
	
	\title{Construction of Weighted Clifford Prolate Spheroidal wave Functions}
	
	\author[H. Baghal Ghaffari]{Hamed Baghal Ghaffari}
	\address{School of Mathematical and Physical Sciences \\University of Newcastle \\Callaghan\\ NSW 2308\\Australia}
	\email{hamed.baghalghaffari@newcastle.edu.au}
	\author[S. Bernstein]{Swanhild Bernstein}
	\address{Institute for Applied Analysis\\Faculty of Mathematics and Computer Science \\Technische Universit{\"a}t Bergakademie Freiberg\\ 09596 Freiberg\\Germany}
	\email{bernstei@math.tu-freiberg.de}

	\keywords{Weighted Clifford Prolate Spheroidal Wave Functions, Bonnet Formula, Clifford Analysis
		\newline
		AMS Classification is $15A66$, $15A67$
	}

	\begin{abstract}
		We develop some properties and the Bonnet formula for Clifford Gegenbauer polynomials. Then after we define and construct weighted Clifford prolate spheroidal wave functions. We then prove that they are orthogonal in a weighted function space.
	\end{abstract}
	
	\bigskip
	\maketitle
	\section{Introduction}
	
	 In 1961, \cite{slepian1961prolate} Slepian, Pollak and Landau observed that Prolate spheroidal wave functions (PSWFs) which have long been used in mathematical physics were the solutions of the spectral concentration problem which is of enormous importance in communications technologies. In 1964, \cite{slepian1964prolate} the multidimensional version of PSWFs as one of the earliest generalizations has been developed. However, the  mentioned generalized PSWFs had singularity in the related differential equations causing numerical instability. Moreover, the multidimensional PSWFs were at most complex-valued, i.e., not multichannel, decreasing the flexibility in applications such as image processing. Because of these two problems, the Clifford higher dimension version of  PSWFs, i.e., Clifford PSWFs (CPSWFs) has been developed in \cite{ghaffari2019clifford,ghaffari2021clifford, ghaffari2022higher}. In 2010 the weighted version of the one-dimensional PSWFs was constructed by Li-Lian Wang  and Jing Zhang \cite{wang2010new} for approximation applications to the class of functions belonging to the weighted function space. A multidimensional  version of the latter work in higher dimensions was also developed in \cite{zhang2020ball}. Some applications and properties of (weighted) PSWFs can be found in
	\cite{thomson2007jackknifing,chen2008mimo,dilmaghani2003novel,gosse2013compressed,hu2014doa,hogan2010sampling,khare2003sampling,lindquist2008spatial,senay2009reconstruction,boulsane2019mean,jaming2020almost}.
	
	Here we continue the work developed in \cite{ghaffari2021clifford} with the higher dimensional version of the weight given in \cite{wang2010new}. For constructing weighted CPSWFs we find a Bonnet formula for Clifford Gegenbauer polynomials (CGPs). In fact, the Bonnet formula is the key and the innovation part for getting into the construction of weighted CPSWFs and proving they are orthogonal in related space.

	This paper is organized as follows. The second section gives background related to Clifford analysis and Jacobi polynomials. In section 3 we recall the definition of the CGPs and investigate some of their properties such as the Bonnet formulas. In section 4 we define the weighted CPSWFs and then we compute them. In section 5, we find the Sturm-Liouville properties of the weighted CPSWFs which enable us to say they are a basis for the weighted function space.
	\section{Background}
	Let
	$\mathbb{R}^{m}$
	be 
	$m$-dimensional 
	euclidean space and let
	$\{e_{1},e_{2},\dots ,e_{m}\}$
	be an orthonormal basis for
	$\mathbb{R}^{m}.$
	We endow these vectors with the multiplicative properties
	\begin{align*}
		e_{j}^{2}&=-1,\; \; j=1,\dots , m,\\
		e_{j}e_{i}&=-e_{i}e_{j}, \;\; i\neq j, \;\; i,j=1,\dots , m.
	\end{align*}
	For any subset
	$A=\{j_{1},j_{2},\dots, j_{h}\}\subseteq \{1,\dots ,	m\}=Q_m,$ with $j_1<j_2<\cdots <j_h$
	we consider the formal product
	$e_{A}=e_{j_{1}}e_{j_{2}}\dots e_{j_{h}}.$
	Moreover for the empty set
	$\emptyset$
	one puts
	$e_{\emptyset}=1$ (the identity element). The Clifford algebra ${\mathbb R}_m$ is then the $2^m$-dimensional associative algebra 
	$${\mathbb R}_m=\bigg\{\sum\limits_{A\subset Q_m}\lambda_Ae_A:\, \lambda_A\in{\mathbb R}\bigg\}.$$
	Every element $\lambda=\sum\limits_{A\subset Q_m}\lambda_Ae_A\in{\mathbb R}_m$ may be decomposed as  
	$\lambda=\sum\limits_{k=0}^{m}[\lambda]_{k},$
	where 
	$[\lambda]_{k}=\sum\limits_{\vert A\vert=k}\lambda_{A}e_{A}$
	is the so-called 
	$k$-vector
	part of 
	$\lambda\, (k=0,1,\dots ,m).$
	
	Denoting by 
	$\mathbb{R}_{m}^{k}$
	the subspace of all 
	$k$-vectors
	in
	$\mathbb{R}_{m},$
	i.e., the image of 
	$\mathbb{R}_{m}$
	under the projection operator 
	$[\cdot]_{k},$
	one has the multi-vector decomposition
	$\mathbb{R}_{m}=\mathbb{R}_{m}^{0}\oplus \mathbb{R}_{m}^{1}\oplus\cdots \oplus \mathbb{R}_{m}^{m},$
	leading  to the identification of
	$\mathbb{R}$
	with the subspace of real scalars
	$\mathbb{R}_{m}^{0}$
	and of
	$\mathbb{R}^{m}$
	with the subspace of real Clifford vectors 
	$\mathbb{R}_{m}^{1}.$ The latter identification is achieved by identifying the point
	$(x_{1},\dots,x_{m})\in{\mathbb R}^m$
	with the Clifford number
	$x=\sum\limits_{j=1}^{m}e_{j}x_{j}\in{\mathbb R}_m^1$.
	The Clifford number 
	$e_{M}=e_{1}e_{2}\cdots e_{m}$
	is called the pseudoscalar; depending on the dimension 
	$m,$
	the pseudoscalar commutes or anti-commutes with the 
	$k$-vectors
	and squares to 
	$\pm 1.$
	The Hermitian conjugation is the real linear mapping $\lambda\mapsto\bar{\lambda}$ of ${\mathbb R}_m$ to itself satisfying
	\begin{align*}
		\overline{\lambda \mu}&=\bar{\mu}\bar{\lambda},\;\;\;\; \textnormal{for all}\;\lambda,\mu\in\mathbb{R}_{m}\\
		\overline{\lambda_{A}e_{A}}&=\lambda_{A}\overline{e_{A}},\;\;\; \lambda\in\mathbb{R},\\
		\overline{e_{j}}&=-e_{j},\;\; j, \;\; j=1,\cdots , m.
	\end{align*}
	The Hermitian conjugation leads to a Hermitian inner product $(\cdot ,\cdot )$ and its associated norm $|\cdot |$ on 
	$\mathbb{R}_{m}$
	given respectively by
	$$(\lambda, \mu)=[\bar{\lambda}\mu]_{0}\;\;\;\textnormal{and}\;\;\; \vert\lambda\vert^{2}=[\bar{\lambda}\lambda]_{0}=\sum\limits_{A}\vert\lambda_{A}\vert^{2}.$$
	The product of two vectors splits up into a scalar part and a 2-vector, also called a bivector:
	$$xy=-\langle x, y\rangle +x\wedge y$$
	where
	$\langle x,y\rangle=-\sum\limits_{j=1}^{m}x_{j}y_{j}\in \mathbb{R}^{0}_{m}$,
	and
	$x\wedge y=\sum\limits_{i=1}^{m}\sum\limits_{j=i+1}^{m}e_{i}e_{j}(x_{j}y_{j}-x_{j}y_{i})\in\mathbb{R}^{2}_{m}$.
	Note that the square of a vector variable 
	$x$
	is scalar-valued and equals the norm squared up to minus sign:
	$$x^{2}=-\langle x,x\rangle=-\vert x\vert^{2}.$$
	Clifford analysis offers a function theory which is a higher-dimensional analogue of the theory of holomorphic functions of one complex variable. The functions considered are defined in the Euclidean space 
	$\mathbb{R}^{m}$
	and take their values in the Clifford algebra 
	$\mathbb{R}_{m}.$
	
	The central notion in Clifford analysis is monogenicity, which is a multidimensional counterpart of holomorphy in the complex plane. 
	Let $\Omega\subset{\mathbb R}^m$, $f:\Omega\to{\mathbb R}_m$ and $n$ a non-negative integer. We say $f\in C^n(\Omega,{\mathbb R}_m )$ if $f$ and all its partial derivatives of order less than or equal to $n$ are continuous.
	\begin{Definition} Let $\Omega\subset{\mathbb R}^m$. 
		A function 
		$f\in C^1(\Omega ,{\mathbb R}_m)$
		is said to be left monogenic in that region if 
		$$\partial_{x}f=0.$$
		Here 
		$\partial_{x}$
		is the Dirac operator in
		$\mathbb{R}^{m}$, i.e.,
		$\partial_{x}f=\sum\limits_{j=1}^{m}e_{j}\partial_{x_{j}}f$,
		where
		$\partial_{x_{j}}$
		is the partial differential operator
		$\dfrac{\partial}{\partial x_{j}}.$
		The Euler operator is defined on  
		$C^1(\Omega ,\mathbb{R}_{m})$
		by
		$E=\sum\limits_{j=1}^{m}x_{j}\partial_{x_{j}}$.
		If $k$ is a non-negative integer and $f\in C^1({\mathbb R}^m\setminus\{0\},{\mathbb R}_m)$ is homogeneous of degree $k$ (i.e., $f(\lambda x)=\lambda^kf(x)$ for all $\lambda >0$ and $x\in{\mathbb R}^m$) then $Ef=kf$.
		The Laplace operator is factorized by the Dirac operator as follows:
		\begin{equation}
			\Delta_{m}=-\partial_{x}^{2}.
		\end{equation}
	\end{Definition}
	The notion of right monogenicity is defined in a similar way by letting the Dirac operator act from the right. It is easily seen that if a Clifford algebra-valued function 
	$f$
	is left monogenic, its Hermitian conjugate 
	$\bar{f}$
	is right monogenic. 
	
	\begin{Th}(\textbf{Clifford-Stokes theorem})\label{Clifford-Stokes theorem}
		Let $\Omega\subset{\mathbb R}^m$,
		$f,g\in {C}^{1}(\Omega ,{\mathbb R}_m)$
		and assume that 
		$C$
		is a compact orientable 
		$m-$dimensional sub-manifold of $\Omega$ with boundary
		$\partial(C)$.
		Then 
		$$\int\limits_{\partial C}f(x)n(x)g(x)d\sigma(x)=\int\limits_{C}[(f(x)\partial_{x})g(x)+f(x)(\partial_{x}g(x))]dx.$$
		where 
		$n(x)$
		is the outward-pointing unit normal vector on
		$\partial C.$
	\end{Th}
	\begin{proof}
		For a proof, see
		\cite{delanghe2012clifford}
	\end{proof}
	As a consequence of Theorem \ref{Clifford-Stokes theorem}, we have the following result.
	\begin{Corollary}\label{property_two_monogenic_x_between}
		Let $f$, $g$ be defined on a neighbourhood $\Omega $ of the unit ball in ${\mathbb R}^m$ and suppose $f$ is right monogenic on $\Omega$ while $g$ is left monogenic on $\Omega$.
		Then
		\begin{equation}
			\int\limits_{S^{m-1}}f(\omega )\omega g(\omega )\, d\omega =0.
		\end{equation}
	\end{Corollary}
	
	\begin{Definition}\label{left monogenic homogeneous polynomial}
		A left monogenic homogeneous polynomial
		$Y_{k}$
		of degree 
		$k\; (k\geq 0)$
		on
		$\mathbb{R}^{m}$
		is called a left solid inner spherical monogenic of order 
		$k.$
		The set of all left solid inner spherical monogenics of order 
		$k$
		will be denoted by
		$M_{l}^{+}(k).$
		It can be shown 
		\cite{delanghe2012clifford}
		that the dimension of 
		$M_{l}^{+}(k)$
		is given by 
		$$\dim M_{l}^{+}(k)=\frac{(m+k-2)!}{(m-2)!k!}=\binom{m+k-2}{k}=d_{k,m}.$$
		We can choose an orthonormal basis 
		for each 
		$M_{l}^{+}(k)$, $(k\geq 0)$ i.e., a collection $\{Y_{k}^{j}\}_{j=1}^{d_{k,m}}$ which spans $M_l^+(k)$ and for which 
		$$\int\limits_{S^{m-1}}\overline{Y_{k}^{j}(\omega )}Y_{k}^{j'}(\omega)\, d\omega=\delta_{jj'}.$$
	\end{Definition}
	
	\begin{Lemma}\label{lem: D and Delta on Y_k}
		For
		$Y_{k}\in M_{l}^{+}(k)$
		and 
		$s\in \mathbb{N}$,
		the following fundamental formula hold:
		\begin{equation}
			\partial_{x}[x^{s}Y_{k}(x)]=
			\begin{cases}
				-sx^{s-1}Y_{k}(x)&\text{ for $s$ even}\\
				-(s+2k+m-1)x^{s-1}Y_{k}(x)&\text{ for $s$ odd}
			\end{cases}
		\end{equation}
		and for $s\geq 2,$
		\begin{equation}
			\Delta_{m}[x^{s}Y_{k}(x)]=
			\begin{cases}
				-s(s+2k+m-2)x^{s-2}Y_{k}(x)&\text{ for $s$ even}\\
				-(s+2k+m-2)(s-1)x^{s-2}Y_{k}(x)&\text{ for $s$ odd.}
			\end{cases}
		\end{equation}
	\end{Lemma}
	\begin{Definition}
		A real-valued polynomial
		$S_{k}$
		of degree
		$k$
		on
		$\mathbb{R}^{m}$
		satisfying
		$$\Delta_{m}S_{k}(x)=0\;\;\;\;\; \textnormal{and}\;\;\;\;\; S_{k}(tx)=t^{k}S_{k}(x)\quad (t>0),$$
		is called a solid spherical harmonic of degree $k$. The collection of solid spherical harmonics of degree $k$ on
		$\mathbb{R}^{m}$
		is denoted
		$\mathcal{H}(k)$
		(or $\mathcal{H}(m,k)$).
	\end{Definition}
	
	Since
	$\partial^{2}_{x}=-\Delta_{m}$,
	we have that
	$$M_{l}^{+}(k)\subset \mathcal{H}(k),\;\;\;\;\; \textnormal{and,}\;\;\;\;\; M_{r}^{+}(k)\subset \mathcal{H}(k).$$
	Let
	$H_{(r)}$
	be a unitary right Clifford-module, i.e. 
	$(H_{(r)},+)$
	is an abelian group
	and a law of scalar multiplication
	$(f,\lambda)\to f\lambda$
	from
	$H_{(r)}\times\mathbb{C}_{m}$
	into
	$H_{r}$
	is defined such that for all
	$\lambda,\mu\in\mathbb{C}_{m}$
	and
	$f,g\in H_{(r)}:$
	\begin{align*}
		&(i)\;f(\lambda+\mu)=f\lambda+f\mu,\hspace*{10cm}\\
		&(ii)\;f(\lambda\mu)=(f\lambda)\mu,\\
		&(iii)\;(f+g)\lambda=f\lambda+g\lambda,\\
		&(iv)\;fe_{\emptyset}=f.
	\end{align*}
	Note that
	$H_{(r)}$
	becomes a complex vector space if
	$\mathbb{C}$
	is identified with
	$\mathbb{C}e_{\emptyset}\subset\mathbb{C}_{m}.$
	Then a function
	$\langle\cdot,\cdot\rangle:H_{(r)}\times H_{(r)}\to\mathbb{C}_{m}$
	is said to be an inner product on
	$H_{(r)}$
	if for all
	$f,g,h\in H_{(r)}$
	and
	$\lambda\in\mathbb{C}_{m}:$
	\begin{align*}
		&(i)\;\langle f,g\lambda+h\rangle=\langle f,g\rangle\lambda+\langle f,h\rangle,\hspace*{10cm}\\
		&(ii)\;\langle f,g\rangle=\overline{\langle g,f\rangle},\\
		&(iii)\;[\langle f,f\rangle]_{0}\geq 0 \;\; \textnormal{and}\;\; [\langle f,f\rangle]_{0}= 0 \;\;\textnormal{if and only if}\;\; f=0.
	\end{align*}
	From this
	${\mathbb C}_m$-valued
	inner product
	$\langle\cdot,\cdot\rangle$,
	one can recover the complex inner product
	$$( f,g) =[\langle f,g\rangle]_{0},$$
	on
	$H_{r}$. Putting for each
	$f\in H_{(r)}$
	\begin{equation}
		\Vert f\Vert^{2}=( f,f) ,\label{ip norm}
	\end{equation}
	$\Vert \cdot\Vert$
	becomes a norm on 
	$H_{r}$
	turning it into a normed right Clifford-module. 
	
	Now, let $H_{(r)}$
	be a unitary right Clifford-module provided with an inner product
	$\langle\cdot,\cdot\rangle.$
	Then it is called a right Hilbert Clifford-module if
	$H_{(r)},$ considered as a complex vector space provided with the complex inner product 
	$(\cdot,\cdot),$ is a Hilbert space. 
	
	We consider the weighted ${\mathbb C}_m$-valued inner product of the functions 
	$f,g:{\mathbb R}^m\to{\mathbb C}_m$ by
	$$ \langle f,g\rangle_{\alpha}=\int\limits_{\mathbb{R}^{m}}\overline{f(x)}g(x)\,(1-\vert x\vert^{2})^{\alpha} dx,$$
	where
	$dx$
	is Lebesgue measure on 
	$\mathbb{R}^{m}$
	and moreover the associated norm given by \eqref{ip norm}.
	The weighted right Clifford-module of ${\mathbb C}_m$-valued measurable functions on
	$\mathbb{R}^{m}$ for which $\Vert f\Vert^{2}<\infty$
	is a weighted right Hilbert Clifford-module which we denote by 
	$L^{2}_{\alpha}(\mathbb{R}^{m},{\mathbb C}_m).$
	Therefore, we obtain the weighted right Hilbert Clifford-module of square integrable functions:
	$L^{2}_{\alpha}(\mathbb{R}^{m},{\mathbb C}_m)$ of functions $f:{\mathbb R}^m\to{\mathbb C}_m$ for which each component of $f$ is measurable and 
	$$\Vert f\Vert_{L^{2}_{\alpha}}=\left(\int\limits_{\mathbb{R}^{m}}\vert f(x)\vert^{2}\, (1-\vert x\vert^{2})^{\alpha} dx\right)^{\frac{1}{2}}<\infty.$$
	
	Now we introduce a Clifford decomposition of functions in $L^{2}_{\alpha}(\mathbb{R}^{m},\mathbb{C}_{m})$. A proof can be found in \cite{delanghe2012clifford}.
	
	\begin{Th}\label{Representation_f_all_monogenics}
		For every function $f\in L^{2}_{\alpha}(\mathbb{R}^{m}, \mathbb{C}_{m}),$ we have the following representation 
		\begin{equation}\label{Representation_f_all_monogenics_equation}
			f(x)=\sum_{k=0}^{\infty}\sum_{\ell=1}^{d_{k,m}}\left[f_{k}^{\ell}(\vert x\vert)Y_{k}^{\ell}(x)+g_{k}^{\ell}(\vert x\vert)\frac{x}{\vert x\vert^{m}}Y_{k}^{\ell }\bigg(\frac{x}{\vert x\vert^2}\bigg)\right],
		\end{equation}
		where $Y_{k}^{\ell}\in M_{l}^{+}(k)$ and the radial functions $f_k^{\ell }$ and $g_k^{\ell }$ satisfy
		$$\int\limits_{0}^{\infty} r^{m+2k-1}\vert f_{k}^{\ell}(r)\vert^{2}(1-r^{2})^{\alpha}dr<\infty,\quad\int\limits_{0}^{\infty} r^{1-m-2k}\vert g_{k}^{\ell}(r)\vert^{2}(1-r^{2})^{\alpha} dr<\infty.$$
	\end{Th}
	The proof of the following lemma for $\alpha=0$ can be found in \cite{ghaffari2021clifford}. The proof for $\alpha\neq 0$ can be obtained easily. 
	\begin{Lemma}\label{support_Lemma_f_summand}
		Let $f\in L^{2}_{\alpha}(\mathbb{R}^{m},\mathbb{C}_{m}),$ be supported on ${B(1)}$. Then the radial functions $f_{k}^{\ell}$ and $g_{k}^{\ell}$ defined in Theorem \ref{Representation_f_all_monogenics} are supported on the interval $[0,1]$.
	\end{Lemma}
	
	The Jacobi polynomials $P_n^{(\alpha ,\beta )}(x)$ ($n$ a nonnegative integer, $\alpha ,\beta >-1$, $x\in[-1,1]$) satisfy the differential equation
	\begin{equation}\label{jacobi_differential_equation}
	(1-x^2)y''+[\beta -\alpha -(\alpha +\beta+2)x]y'=-n(n+\alpha +\beta +1)y.
	\end{equation}
	The eigenvalues $-n(n+\alpha +\beta +1)$ are non-degenerate. The Jacobi polynomials admit the explicit representation
	\begin{equation}\label{jacobi_explicit_representation}
		P_n^{(\alpha ,\beta )}(x)=\frac{\Gamma(\alpha+n+1)}{n! \Gamma(\alpha+\beta+n+1)}\sum_{s=0}^n\binom{n}{s}\frac{\Gamma(\alpha+\beta+n+s+1)}{\Gamma(\alpha+s+!)}(\frac{x-1}{2})^s.
	\end{equation}
	
	We now review relevant facts from the theory of Sturm-Liouville problems, proofs of which can be found in \cite{christensen2010functions,al2008sturm} 
	\begin{Definition}\label{def_SL}
		Let $p,\, q,$ and $ r $ be functions on $\mathbb{R}$ or a subinterval of $\mathbb{R};$ assume that the function $ p $ is differentiable, and that the functions $p, q$ and $ r $ are continuous. A differential equation that can be written on the form
		\begin{equation}\label{SL_form}
			[p(x)u']'+[q(x)+\lambda r(x)]u=0,
		\end{equation}
		for some parameter 
		$\lambda\in \mathbb{R},$
		is called a Sturm-Liouville differential equation and a solution $u$ of \eqref{SL_form} is known as an eigenfunction of the Sturm-Liouville form.
	\end{Definition}
	\begin{Th}\label{firstSL}
		Let $u_n$ and $u_m$ be real-valued eigenfunctions for a  Sturm-Liouville form defined in \eqref{SL_form} corresponding to different eigenvalues $\lambda_{n}, \,\lambda_{m}.$ Let $p(x)>0$ and $r(x)>0$ for all $x\in (a,b),$ and also let $p(a)=p(b)=0.$
		Then $u_{n},$ and, $u_{m}$ are orthogonal in $L_{r}^{2}(a,b)$, i.e., 
		$$\int\limits_a^br(x)u_n(x)u_m(x)\, dx=0.$$
	\end{Th}
	\begin{Th}\label{secondSL}
		Let $p_i, q_i,\, i = 1, 2,$ be real-valued continuous functions on the interval $[a, b]$ and let
		$$(P_{1}(x)y')'+q_{1}(x)y=0,$$
		$$(P_{2}(x)y')'+q_{2}(x)y=0,$$
		be two homogeneous linear second order differential equations in self-adjoint form with
		$0<P_{2}(x)\leq P_{1}(x),$
		and,
		$q_{1}(x)\leq q_{2}(x).$
		Let $u$ be a non-trivial solution of the first of these equations with successive roots at $z_1$ and $z_2$ and let $v$ be a non-trivial solution of second equation. Then, there exists an $x \in (z_1, z_2)$ such that $v(x) = 0.$
	\end{Th}
	
	\section{Properties of Clifford Gegenbauer Polynomials}
	In this section we give some results related to the CGPs. The CGPs can be obtained (\cite{delanghe2012clifford} p. 287) by the 
	\begin{equation}
		C_{n,m}^{\alpha}(Y_{k})(x)=(1-\vert x\vert^{2})^{-\alpha}\partial_{x}^{n}((1-\vert x\vert^{2})^{\alpha+n}Y_{k}(x))
	\end{equation}
	which also satisfied in the following Clifford differential equation
	\begin{equation}
		(1-\vert x\vert^{2})^{-\alpha}\partial_{x}((1-\vert x\vert^{2})^{\alpha+1}\partial_{x}C_{n,m}^{\alpha}(Y_{k})(x))=C(\alpha,n,m,k)C_{n,m}^{\alpha}(Y_{k})(x),
	\end{equation}
	where
	$$C(\alpha,n,m,k)=\begin{cases}
		n(2\alpha+n+m+2k)&\text{ if $n$ is even,}\\
		(2\alpha+n+1)(n+m+2k-1)&\text{ if
			$n$ is odd.}
	\end{cases}$$
	In other words they are the eigenfunctions of the 
	\begin{equation}\label{Clifford_Gegenbauer_differential_equation}
		\mathcal{L}_{\alpha}f(x)=(1-\vert x\vert^{2})^{-\alpha}\partial_{x}((1-\vert x\vert^{2})^{\alpha+1}\partial_{x}f(x))
	\end{equation}
	for $\alpha>-1$. Now we obtain even and odd Bonnet formula for CGPs.
	From the \cite{gradshteyn2007ryzhik}, equation 8.961 9 with substituting $\beta$ by $\beta+1$  we have that
	\begin{equation}\label{jacobi_first_recurrence}
		(2n+\alpha+\beta+1)P_{n}^{(\alpha,\beta)}(t)=(n+\alpha+\beta+1)P_{n}^{(\alpha,\beta+1)}(t)+(n+\alpha)P_{n-1}^{(\alpha,\beta+1)}(t)
	\end{equation}
	From (\cite{delanghe2012clifford} p. 294) ( $\alpha=0$ version given in Theorem 5.1 in \cite{propertiesofcliffordlegendre}) we have that
	\begin{align}
		&C_{2n,m}^{\alpha}(Y_{k})(x)=(-1)^{n}2^{2n}(\alpha+n+1)_{n}(n)!P_{n}^{(\alpha,k+\frac{m}{2}-1)}(2\vert x\vert^{2}-1)Y_{k}(x)\label{even_version},\\
		&C_{2n+1,m}^{\alpha}(Y_{k})(x)=(-1)^{n+1}2^{2n+1}(\alpha+n+1)_{n+1}(n)!P_{n}^{(\alpha,k+\frac{m}{2})}(2\vert x\vert^{2}-1)xY_{k}(x).\label{odd_version}
	\end{align}
	We multiply \eqref{jacobi_first_recurrence}  by $(-1)^{n}2^{2n}(\alpha+n+1)_{n}n!$ where $(x)_{k}=\frac{\Gamma(x+k)}{\Gamma(x)}$. So we have that 
	\begin{align}
		&(-1)^{n}2^{2n}(\alpha+n+1)_{n}n!(2n+\alpha+\beta+1)P_{n}^{(\alpha,\beta)}(t)\nonumber\\
		&\hspace{5cm}=(-1)^{n}2^{2n}(\alpha+n+1)_{n}n!(n+\alpha+\beta+1)P_{n}^{(\alpha,\beta+1)}(t)\nonumber\\
		&\hspace{5cm} + (-1)^{n}2^{2n}(\alpha+n+1)_{n}n!(n+\alpha)P_{n-1}^{(\alpha,\beta+1)}(t).\label{auxilary_recurrence_even}
	\end{align}
	Now in \eqref{auxilary_recurrence_even} let $\beta=k+\frac{m}{2}-1$ and $t=2\vert x \vert^{2}-1$. By multiplying $xY_{k}(x)$ in \eqref{auxilary_recurrence_even} and using \eqref{even_version} and \eqref{odd_version} we have that 
	\begin{align*}
		(2n+\alpha+k+\frac{m}{2})xC_{2n,m}^{\alpha}(Y_{k})(x)&=\frac{(-1)^{n}2^{2n}(\alpha+n+1)_{n}n! (n+\alpha+k+\frac{m}{2})}{(-1)^{n+1}2^{2n+1}(\alpha+n+1)_{n+1}n!}C_{2n+1,m}^{\alpha}(Y_{k})(x)\nonumber\\
		&+\frac{(-1)^{n}2^{2n}(\alpha+n+1)_{n}n! (n+\alpha)}{(-1)^{n}2^{2n-1}(\alpha+n)_{n}(n-1)!}C_{2n-1,m}^{\alpha}(Y_{k})(x).
	\end{align*}
	Hence we obtain the desired result by simplification as follows
	\begin{equation}\label{even_Bonnet_formulae_Gegenbauer}
		xC_{2N,m}^{\alpha}(Y_{k}^{j})(x)=A_{N,\alpha,k,m}C_{2N+1,m}^{\alpha}(Y_{k}^{j})(x)+B_{N,\alpha,k,m}C_{2N-1,m}^{\alpha}(Y_{k}^{j})(x),
	\end{equation}
	where
	$A_{N,\alpha,k,m}=\dfrac{-(\alpha+k+\frac{m}{2} +N)}{2(\alpha+2N+1)(\alpha+k+\frac{m}{2} +2N)}\;$; $\; B_{N,\alpha,k,m}=\dfrac{2N(\alpha+2N)}{(\alpha+k+\frac{m}{2}+2N)}$,
	in which we can see that by substituting $\alpha=0$ the latter is compatible with Bonnet formula of the Clifford Legendre polynomials (CLPs).
	
	Now for the odd version we can use the following recurrence
	\begin{equation}\label{jacobi_second_recurrence}
		(n+\frac{\alpha}{2}+\frac{\beta}{2}+1)(1+t)P_{n}^{(\alpha,\beta+1)}(t)=(n+1)P_{n+1}^{(\alpha,\beta)}(t)+(n+\beta+1)P_{n}^{(\alpha,\beta)}(t)
	\end{equation}
	and using \eqref{odd_version} to obtain 
	\begin{equation}\label{odd_Bonnet_formulae_Gegenbauer}
		xC_{2N+1,m}^{\alpha}(Y_{k}^{j})(x)=A'_{N,\alpha,k,m}C_{2N+2,m}^{\alpha}(Y_{k}^{j})(x)+B'_{N,\alpha,k,m}C_{2N,m}^{\alpha}(Y_{k}^{j})(x),
	\end{equation}
	$A'_{N,\alpha,k,m}=\dfrac{-(\alpha+N+1)}{2(\alpha+2N+2)(\alpha+k+\frac{m}{2} +2N+1)}\;$; $\; B'_{N,\alpha,k,m}=\dfrac{2(\alpha+2N+1)(k+\frac{m}{2}+N)}{(\alpha+k+\frac{m}{2}+2N+1)}$.
	
	We can obtain $L^{2}((1-\vert x\vert^2)^{-\alpha})$ norm of CGPs by using the orthogonal properties Jacobi polynomials which is 
	\begin{equation}\label{orthogonality_Jacobi}
		\int_{-1}^{1}(1-x)^{\alpha}(1+x)^{\beta}P_{m}^{(\alpha,\beta)}(x)P_{n}^{\alpha,\beta)}(x)\, dx=\delta_{mn}\frac{2^{\alpha+\beta+1}\Gamma(\alpha+n+1)\Gamma(\beta+n+1)}{n!(\alpha+\beta+1+2n)\Gamma(\alpha+\beta+n+1)},
	\end{equation}
	where $\delta_{mn}$ is delta Kronecker and $Re\, \alpha>-1,\; Re\, \beta>-1$. There is another way of proof given in \cite{delanghe2012clifford}.
	\begin{Th}
		The $L^{2}_{\alpha}(B(1),\mathbb C_{m})$ norm of CGPs are given as follows
		\begin{align}
			&\Vert C_{2N,m}^{\alpha}(Y_{k}^{j}) \Vert^{2}_{\alpha}=\frac{2^{4N}(\Gamma(\alpha+2N+1))^{2}\Gamma(k+\frac{m}{2}+N)N!}{2\Gamma(\alpha+N+1)\Gamma(\alpha+k+\frac{m}{2}+N)(\alpha+k+\frac{m}{2}+2N)},\\
			&\Vert C_{2N+1,m}^{\alpha}(Y_{k}^{j}) \Vert^{2}_{\alpha}=\frac{2^{4N+2}(\Gamma(\alpha+2N+2))^{2}\Gamma(k+\frac{m}{2}+N+1)N!}{2\Gamma(\alpha+N+1)\Gamma(\alpha+k+\frac{m}{2}+N+1)(\alpha+k+\frac{m}{2}+2N+1)}.
		\end{align}
	\end{Th}
	\begin{Remark}
		The normalized version of the Bonnet formula of the CGPs can be obtained as follows
		\begin{align}
			x\barC_{2N,m}^{\alpha}(Y_{k}^{j})(x)&=\overline{A_{N,\alpha,k,m}}\,\barC_{2N+1,m}^{\alpha}(Y_{k}^{j})(x)+\overline{B_{N,\alpha,k,m}}\,\barC_{2N-1,m}^{\alpha}(Y_{k}^{j})(x),\label{normalized_even_Bonnet_formulae_Gegenbauer}\\
			x\barC_{2N+1,m}^{\alpha}(Y_{k}^{j})(x)&=\overline{A'_{N,\alpha,k,m}}\,\barC_{2N+2,m}^{\alpha}(Y_{k}^{j})(x)+\overline{B'_{N,\alpha,k,m}}\,\barC_{2N,m}^{\alpha}(Y_{k}^{j})(x),\label{normalized_odd_Bonnet_formulae_Gegenbauer}
		\end{align}
		where $\barC_{n,m}^{\alpha}(Y_{k}^{j})(x)$ is the normalized CGP and
		\begin{align}
			\overline{A_{N,\alpha,k,m}}&=\tfrac{A_{N,\alpha,k,m}\Vert C_{2N+1,m}^{\alpha}(Y_{k}^{j}) \Vert_{\alpha}}{\Vert C_{2N,m}^{\alpha}(Y_{k}^{j}) \Vert_{\alpha}},\label{Coef_1st}\\
			\overline{B_{N,\alpha,k,m}}&=\tfrac{B_{N,\alpha,k,m}\Vert C_{2N-1,m}^{\alpha}(Y_{k}^{j}) \Vert_{\alpha}}{\Vert C_{2N,m}^{\alpha}(Y_{k}^{j}) \Vert_{\alpha}},\label{Coef_2nd}\\
			\overline{A'_{N,\alpha,k,m}}&=\tfrac{A'_{N,\alpha,k,m}\Vert C_{2N+2,m}^{\alpha}(Y_{k}^{j})(x)\Vert_{\alpha}}{\Vert C_{2N+1,m}^{\alpha}(Y_{k}^{j}) \Vert_{\alpha}},\label{Coef_3rd}\\
			\overline{B'_{N,\alpha,k,m}}&=\tfrac{B'_{N,\alpha,k,m}\Vert C_{2N,m}^{\alpha}(Y_{k}^{j}) \Vert_{\alpha}}{\Vert C_{2N+1,m}^{\alpha}(Y_{k}^{j}) \Vert_{\alpha}}.\label{Coef_4th}
		\end{align}
	\end{Remark}
	
	\begin{Definition}
		A weighted tensorial multi-dimensional Fourier transform given by:
		\begin{equation*}
			\mathcal{F}_{\alpha}f(\xi)=\int_{\mathbb{R}^{m}}e^{-2\pi i\langle x,\xi\rangle}f(x) \, (1-\vert x\vert^{2})^{\alpha}\, dx.
		\end{equation*}
	\end{Definition}
	\begin{Th}
		\label{Fourier Transformation of Cl Legendre Polynomial}
		The weighted Clifford Fourier transform of the restriction of the CGPs
		$C_{n,m}^{\alpha}(Y_{k})(x)$
		to the unit ball
		$B(1),$
		is given by
		\begin{equation}
			\mathcal{F}_{\alpha}(C_{n,m}^{\alpha}(Y_{k}))(\xi)=(-1)^{k}i^{n+k}2^{n}\,\Gamma(n+\alpha+1)\,\xi^n\frac{J_{\alpha+k+\frac{m}{2}+n}(2\pi \vert\xi\vert)}{\pi^{\alpha}\vert \xi\vert^{\alpha+k+\frac{m}{2}+n}}Y_{k}(\xi).
		\end{equation}
	\end{Th}
	\begin{proof}
		The proof can be obtained as in \cite{propertiesofcliffordlegendre} Theorem 4.6.
	\end{proof}
	Here we state explicit versions for the CGPs.
	\begin{Th}\label{thm_representationofCGPs}
		Let the $N,k\geq 0$, $\mu=k+\frac{m}{2}$ and $\alpha>-1$. Then
		\begin{align*}
			C_{2N,m}^{\alpha}(Y_{k})(x)&=\frac{2^{2N}\Gamma(\alpha+2N+1)\Gamma(\mu+N)}{\Gamma(\alpha+N+1)\Gamma(\mu+N+\alpha)}\sum\limits_{l=0}^{N}{N\choose l}\frac{\Gamma(l+\mu +N+\alpha)}{\Gamma(l+\mu )}(-1)^{l}\vert x\vert^{2l}Y_{k}(x),\\
			C_{2N+1,m}^{\alpha}(Y_{k})(x)&=-\frac{2^{2N+1}\Gamma(\alpha+2N+2)\Gamma(\mu+N+1)}{\Gamma(\alpha+N+1)\Gamma(\mu+N+\alpha+1)}\sum\limits_{l=0}^{N}{N\choose l}\frac{\Gamma(l+\mu +N+1+\alpha)}{\Gamma(l+\mu +1)}(-1)^{l}\\
			&\qquad\qquad\times\vert x\vert^{2l}xY_{k}(x).		
		\end{align*} 
	\end{Th}
	\begin{proof}
		Using the \eqref{even_version} and \eqref{odd_version} and the explicit representation of Jacobi polynomials given in \eqref{jacobi_explicit_representation} and the fact that $P_{n}^{\alpha,\beta}(-z)=(-1)^{n}P_{n}^{\beta,\alpha}(z)$, the proof can be completed.
	\end{proof}
	In dimension
	$m=2$ we have
	$\dim M_{l}^{+}(k)=\dfrac{(m+k-2)!}{(m-2)!k!}=1$ (see Definition \ref{left monogenic homogeneous polynomial}).
	Consequently, when $m=2$ the function (see e.g. \cite{sommen1989spingroups, lavivcka2012complete, cerejeiras2014generating})
	\begin{equation}
		Y_{k}(r\cos\theta,r\sin\theta)=\dfrac{r^{k}}{\sqrt{2\pi}}[e_1\cos k\theta -e_2\sin k\theta ].
		\label{2d Y_k}
	\end{equation}
	itself forms an orthonormal basis for $M_l^+(k)$ and in this case the CGPs (described explicitly in Theorem \ref{thm_representationofCGPs}) take the form 
	\begin{align*}
		C_{2N,2}^\alpha(Y_k)(x)&=F_{N,k}^1e_1+F^2_{N,k}(x)e_2,\\
		C_{2N+1,2}^\alpha(Y_k)(x)&=G_{N,k}^1+G_{N,k}^2(x)e_{12},
	\end{align*}
	where $F_{N,k}^1$, $F_{N,k}^2$, $G_{N,k}^1$, $G_{N,k}^2$ are real-valued functions defined on the unit ball $B(1)$. In Figures \ref{CGPs_Figure1} and \ref{CGPs_Figure2} below, these functions are plotted for various values of $\alpha,$ $N,$ and, $k$.
	\begin{Remark}
		We can apply the same technique used in \cite{propertiesofcliffordlegendre} by the change of variable $\vert x\vert^{2}=\frac{s+1}{2}$ in \eqref{Clifford_Gegenbauer_differential_equation} for $f(x)=C_{2N,m}^{\alpha}(Y_k)(x)=P_{N,k,m}^{\alpha}(\vert x\vert^{2})Y_{k}^{j}(x)$, to obtain
		\begin{equation}\label{SL_Radial_CGPs}
			(1-s^{2})Q''_{N}(s)+ [ (k+\frac{m}{2} -1 -\alpha ) - s(k+\frac{m}{2} +\alpha + 1) ]Q'_{N}(s)=-\frac{C(\alpha,2N,m,k)}{4}Q_{N}(s),
		\end{equation}
		in which $Q_{N}(s)=P_{N,k,m}^{\alpha}(\vert x\vert^{2})$. The \eqref{SL_Radial_CGPs} is a Jacobi differential equation \eqref{jacobi_differential_equation} which is expected by the \eqref{even_version}. Also, by multiplying a weight function $w(s)=(1+s)^{k+\frac{m}{2}-1}(1-s)^{\alpha}$  on both side of the \eqref{SL_Radial_CGPs}, we can see that $Q_{N}(s)=P_{N,k,m}^{\alpha}(\vert x\vert^{2})$ is in a Sturm-Liouville differential equation form given in \ref{SL_form} . Here the weight is same as  the one in the integral appeared in \eqref{orthogonality_Jacobi}. The odd case can be treated similarly.
	\end{Remark}
	
	\begin{Remark}\label{Relation_even_odd_CLGs_dimension_2}
		Since in dimension $m=2$, we have that $xY_{k}(x)=e_{1}Y_{k+1}(x)$. Therefroe, we can prove that 
		\begin{equation}
			\barC_{2N+1,2}^{\alpha}(Y_{k})(x)=-e_{1}\barC_{2N,2}^{\alpha}(Y_{k+1})(x).
		\end{equation}
		This is visible in figure \ref{CGPs_Figure}.
	\end{Remark}
	
	\begin{figure}
		\centering
		\begin{tabular}{cc}
			\begin{minipage}[h]{0.5\textwidth}
				\includegraphics[width=4cm]{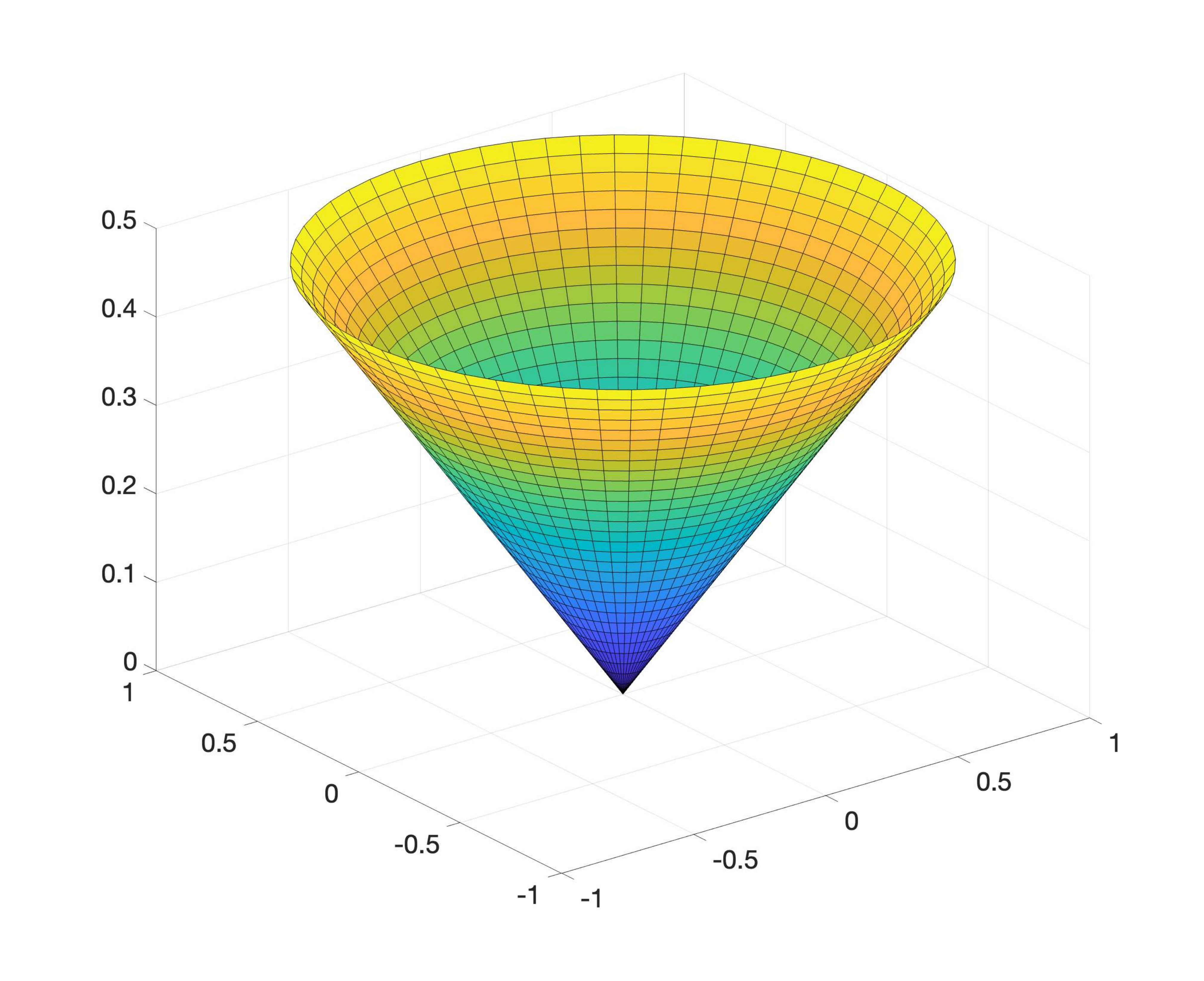}
			\end{minipage} &
			\begin{minipage}[h]{0.5\textwidth}
				\includegraphics[width=4cm]{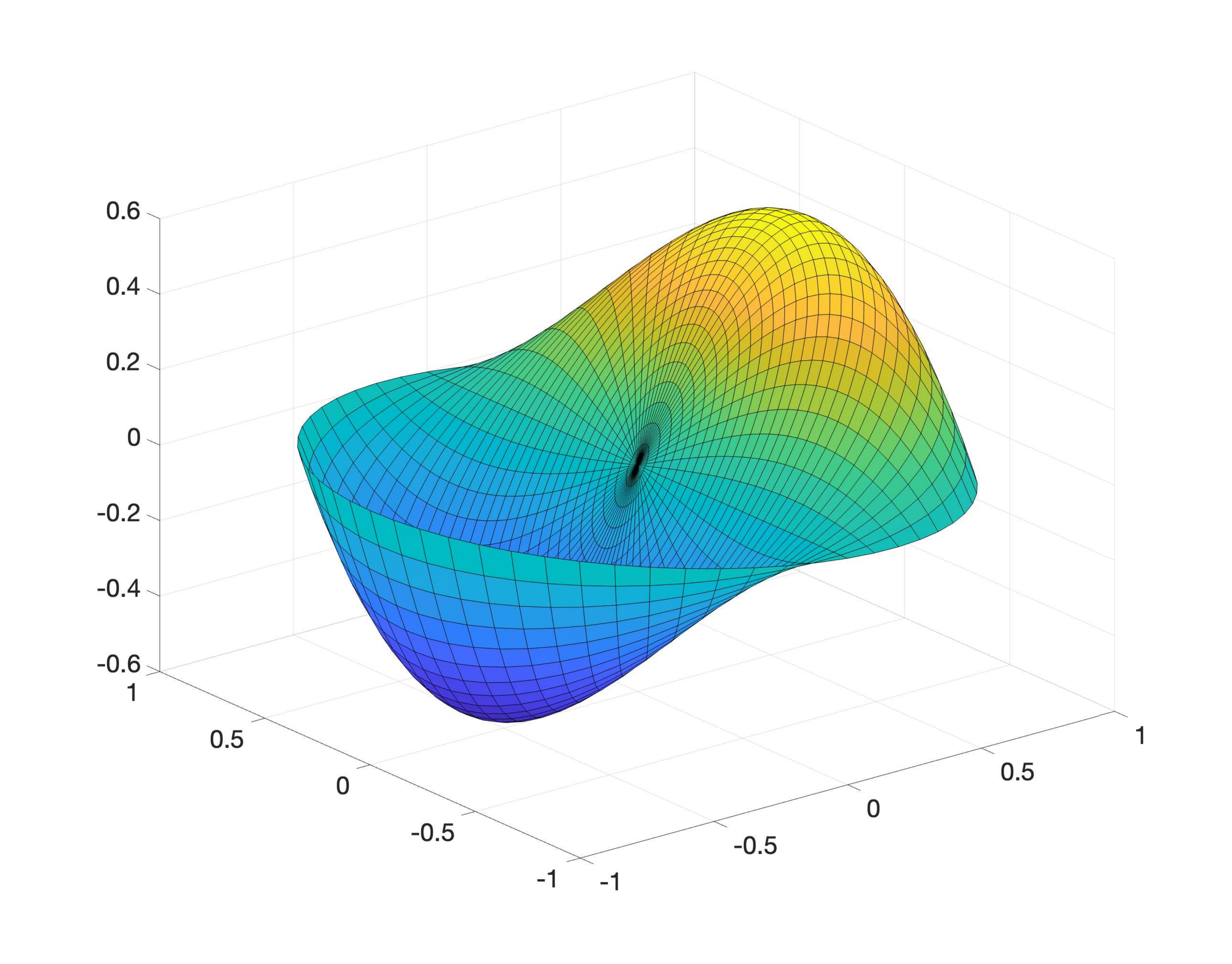}
			\end{minipage}\\
			\begin{minipage}[h]{0.45\textwidth}
				(a) Graph of Radial part of the normalized Clifford-Gegenbauer polynomial $C_{0,2}^{(-0.5)}(Y_1)$.
			\end{minipage} & 
			\begin{minipage}[h]{0.45\textwidth}
				(d) Graph of $e_{1}$ part of the normalized Clifford-Gegenbauer polynomial $C_{2,2}^{(-0.9)}(Y_1)$.
			\end{minipage}\\
			\begin{minipage}[h]{0.5\textwidth}
				\includegraphics[width=4cm]{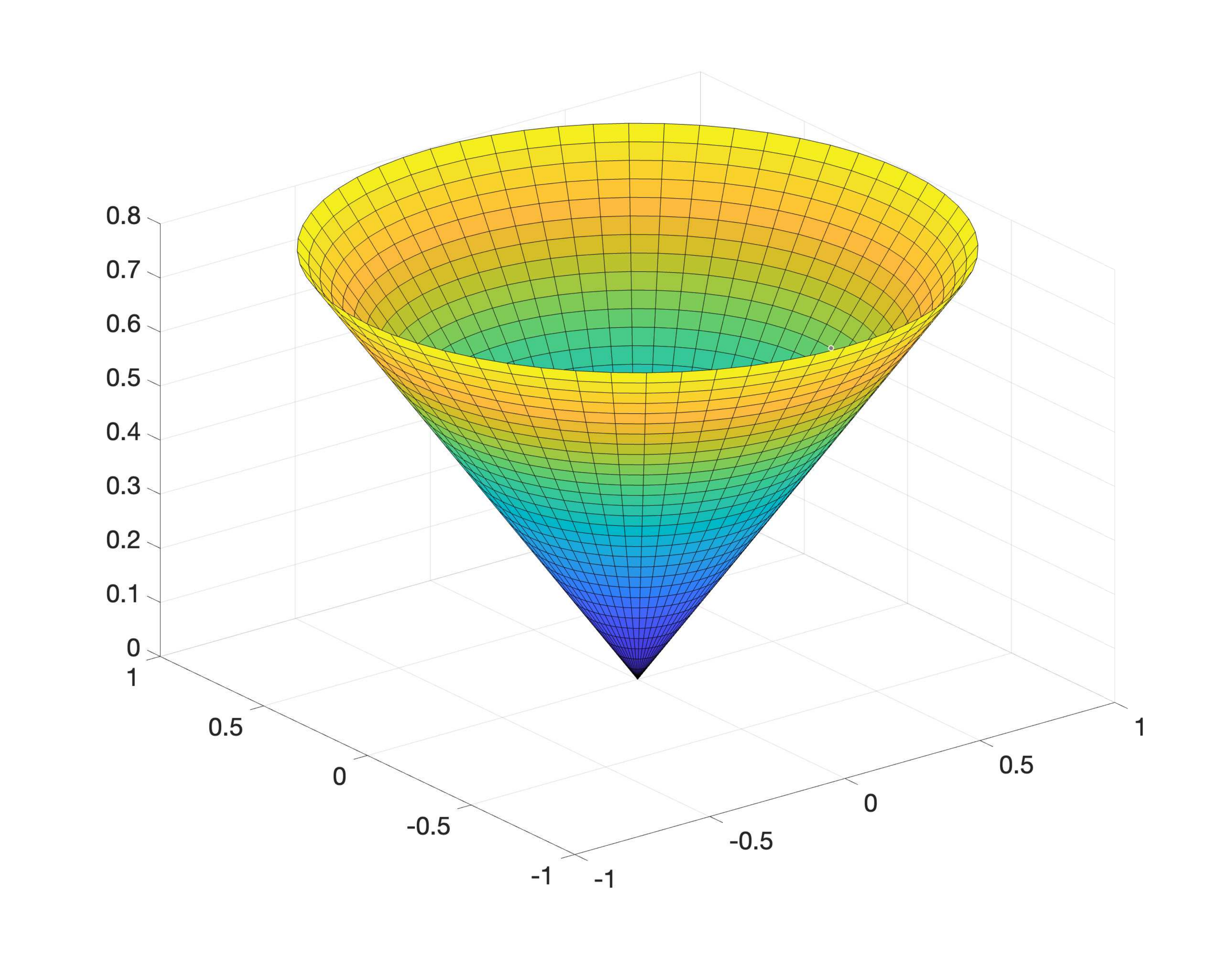}
			\end{minipage}&
			\begin{minipage}[h]{0.5\textwidth}
				\includegraphics[width=4cm]{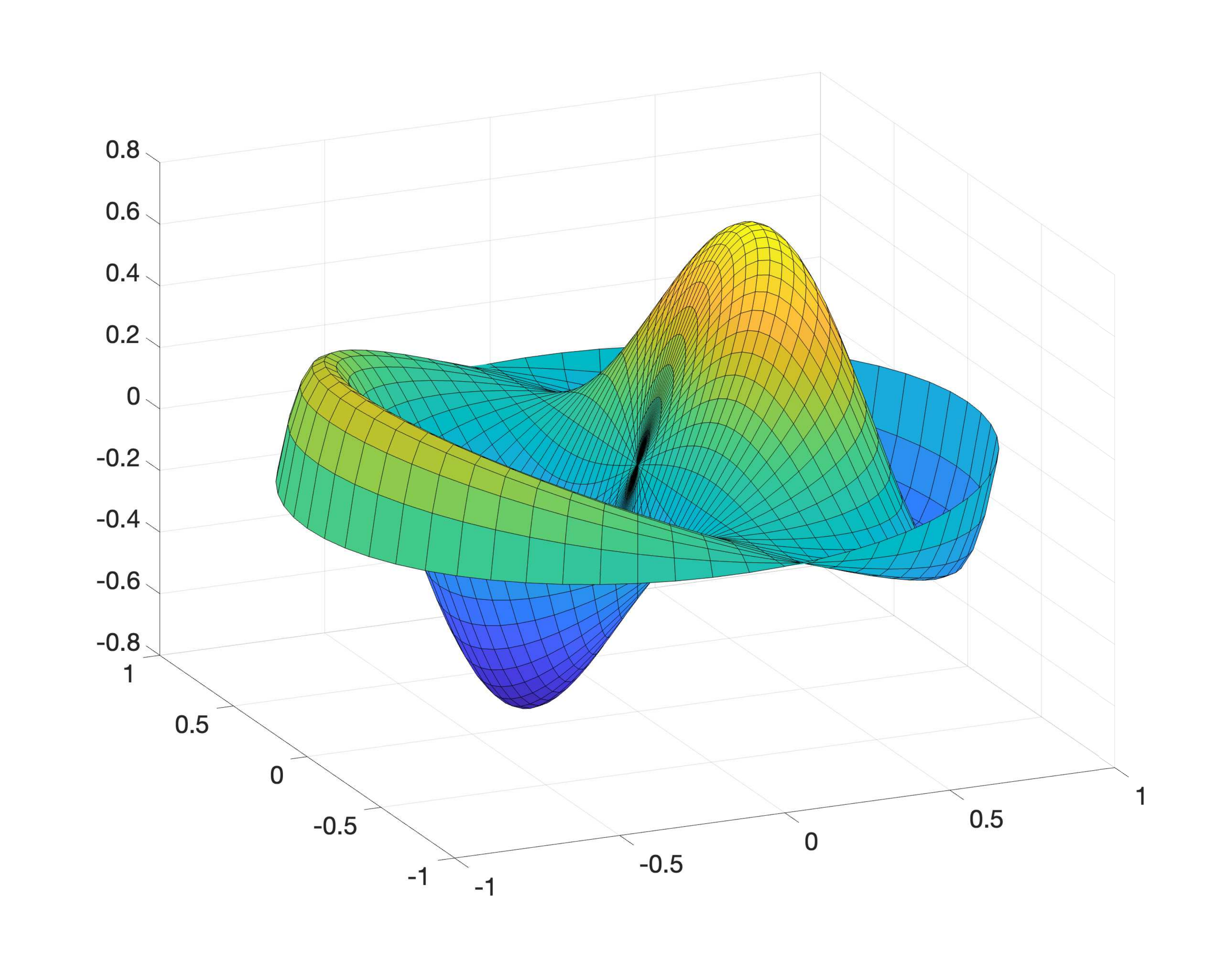}
			\end{minipage}\\
			\begin{minipage}[h]{0.45\textwidth}
				(b) Graph of Radial part of the normalized Clifford-Gegenbauer polynomial $C_{0,2}^0(Y_1)$.
			\end{minipage} &
			\begin{minipage}[h]{0.45\textwidth}
				(e) Graph of $e_{1}$ part of the normalized Clifford-Gegenbauer polynomial $C_{4,2}^{(-0.9)}(Y_1)$.
			\end{minipage}\\		
			\begin{minipage}[h]{0.5\textwidth}
				\includegraphics[width=4cm]{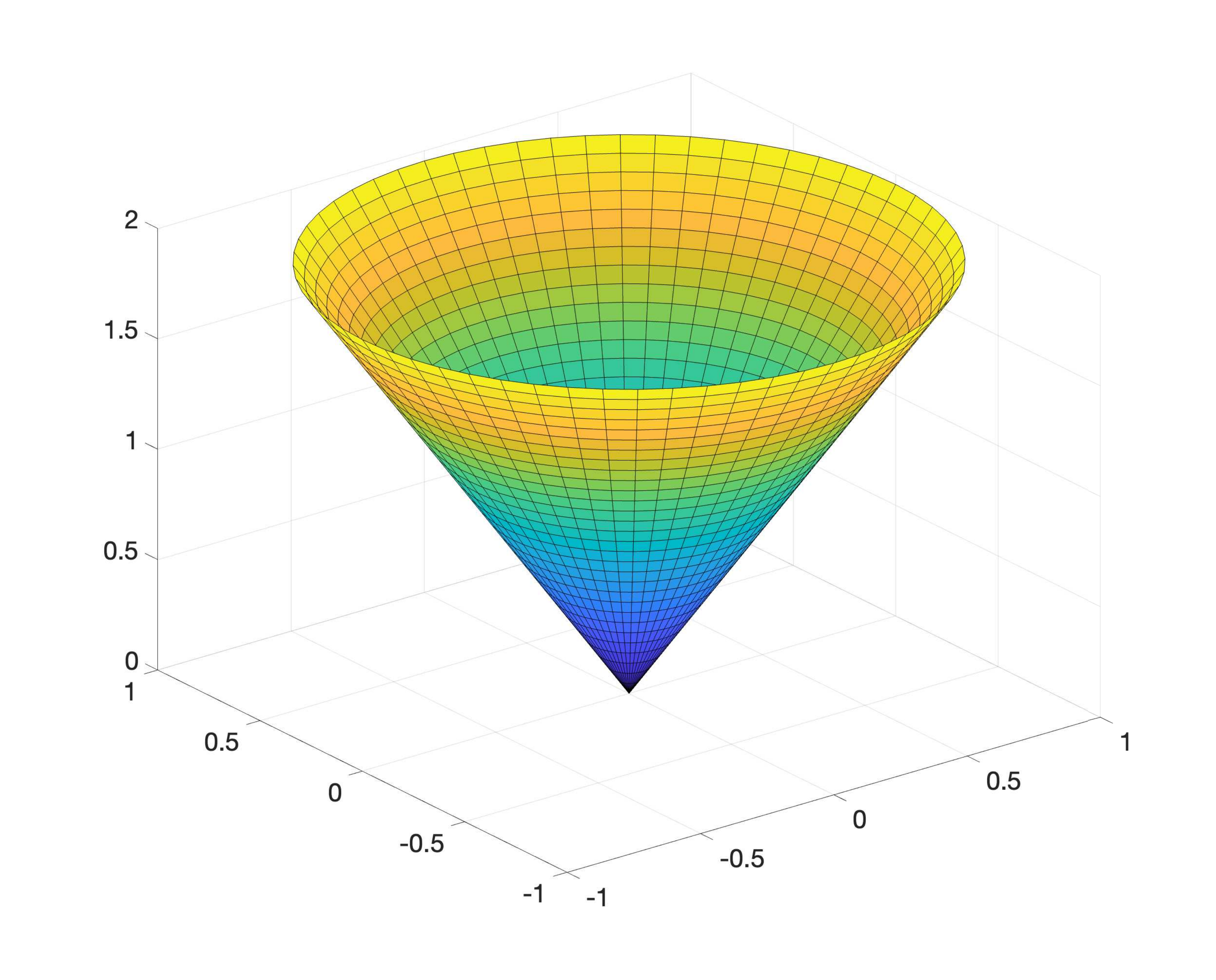}
			\end{minipage}&
			\begin{minipage}[h]{0.5\textwidth}
				\includegraphics[width=4cm]{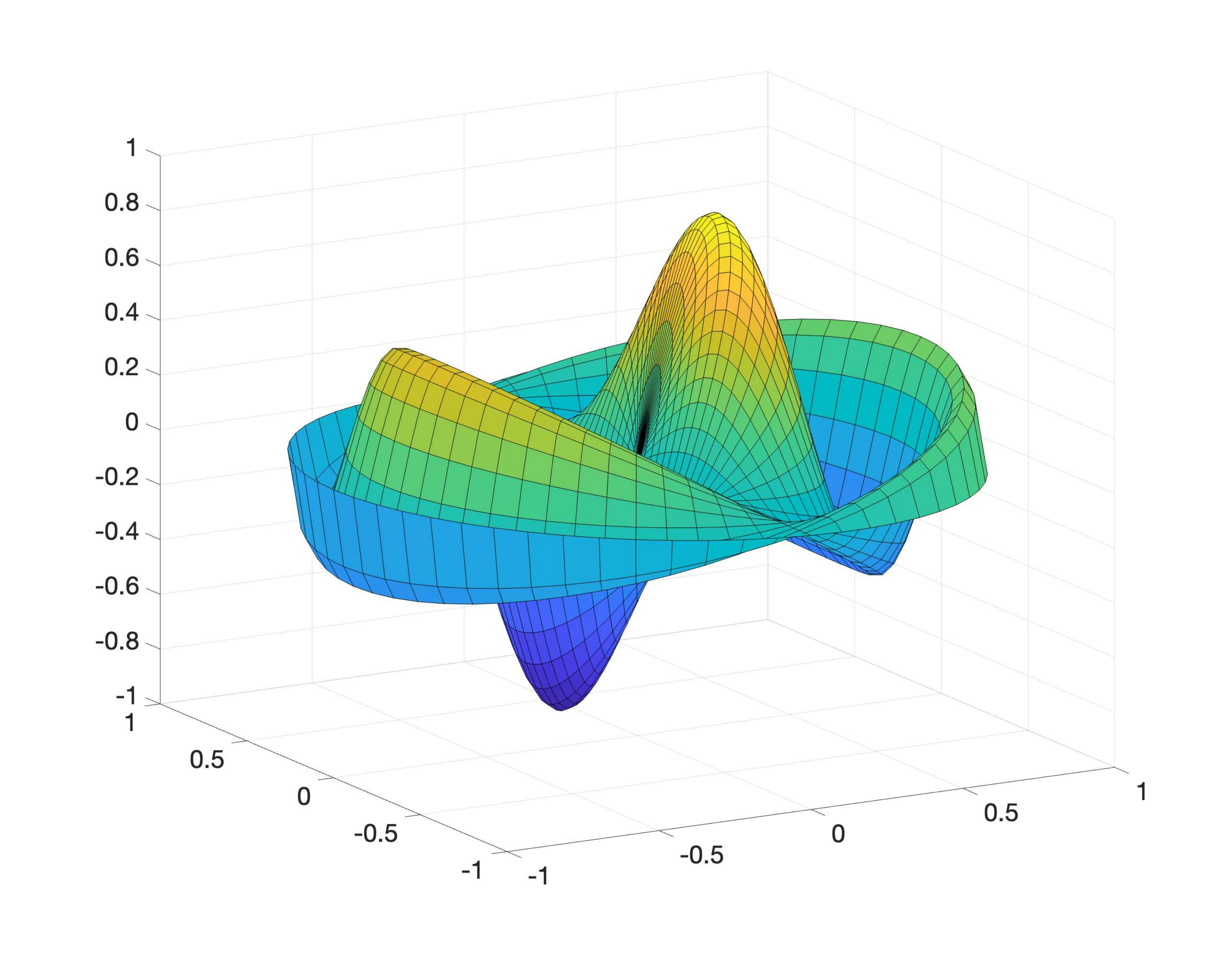}
			\end{minipage}\\
			\begin{minipage}[h]{0.45\textwidth}
				(c) Graph of Radial part of the normalized Clifford-Gegenbauer polynomial $C_{0,2}^2(Y_1)$.
			\end{minipage} &
			\begin{minipage}[h]{0.45\textwidth}
				(f) Graph of $e_{1}$ part of the normalized Clifford-Gegenbauer polynomial $C_{6,2}^{(-0.9)}(Y_1)$.
			\end{minipage}
		\end{tabular}
		\caption{Plots of $C_{2N,2}^\alpha(Y_k)$ }
		\label{CGPs_Figure1}
	\end{figure}

	\begin{figure}
		\centering
		\begin{tabular}{cc}
			\begin{minipage}[h]{0.5\textwidth}
				\includegraphics[width=4cm]{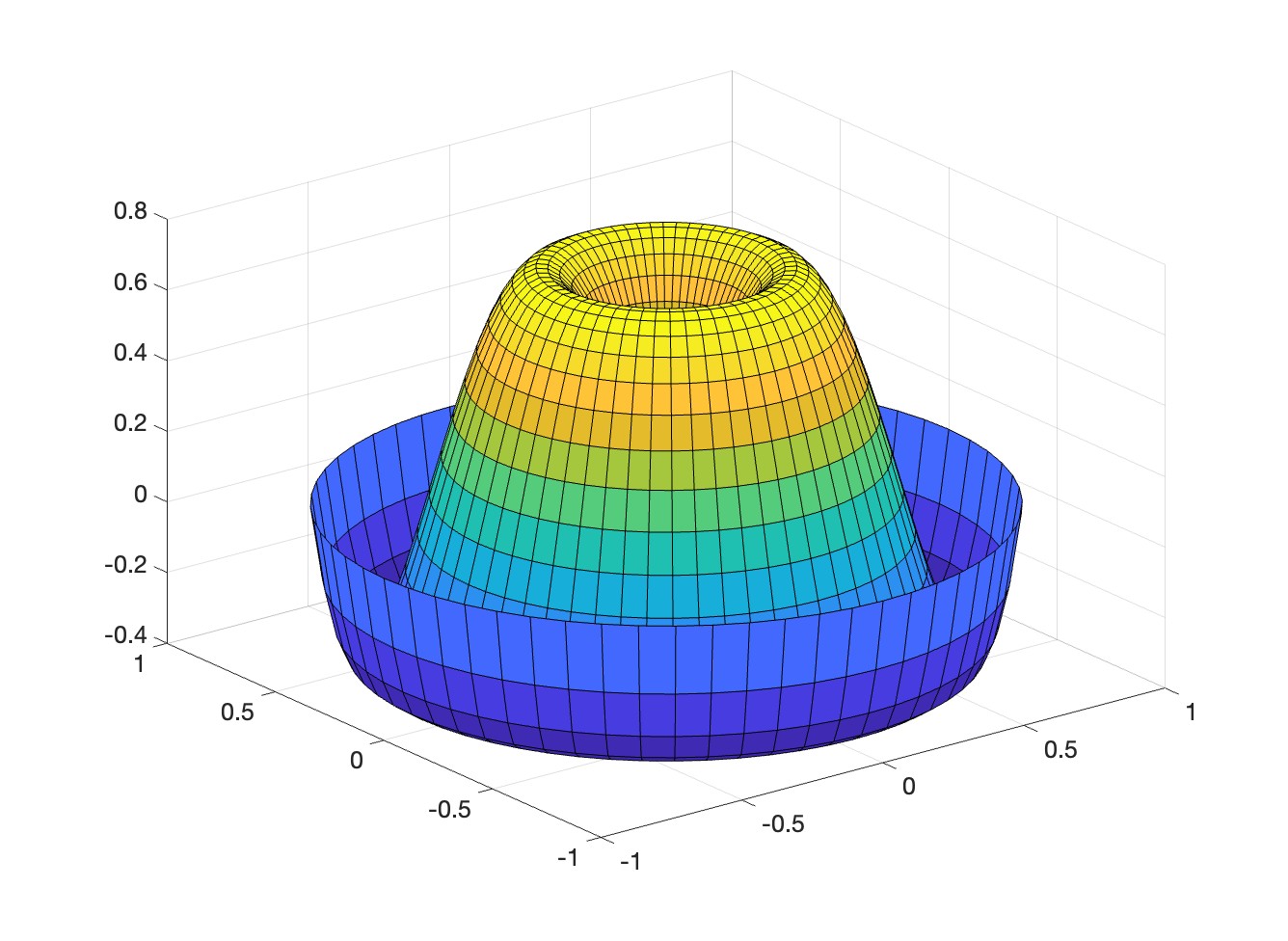}
			\end{minipage} &
			\begin{minipage}[h]{0.5\textwidth}
				\includegraphics[width=4cm]{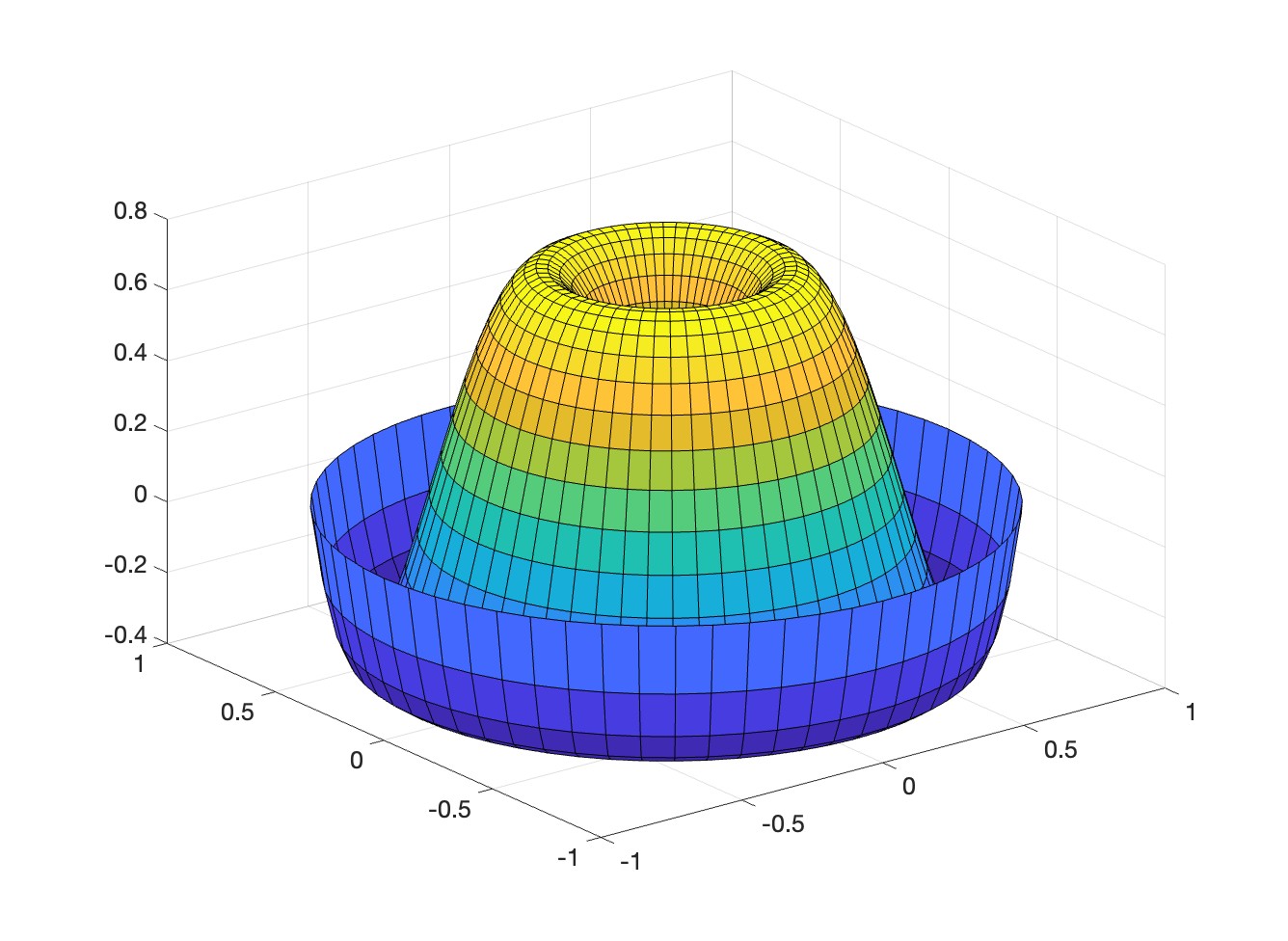}
			\end{minipage}\\
			\begin{minipage}[h]{0.45\textwidth}
				(a) Graph of Radial part of the normalized Clifford-Gegenbauer polynomial $C_{4,2}^{(-0.9)}(Y_1)$
			\end{minipage} & 
			\begin{minipage}[h]{0.45\textwidth}
				(b) Graph of Radial part of the normalized Clifford-Gegenbauer polynomial $C_{5,2}^{(-0.9)}(Y_0)$
			\end{minipage}
		\end{tabular}
		\caption{Plots of $C_{2N,2}^\alpha(Y_k)$.}
		\label{CGPs_Figure}
	\end{figure}

	\begin{figure}
		\centering
		\begin{tabular}{cc}
			\begin{minipage}[h]{0.5\textwidth}
				\includegraphics[width=4cm]{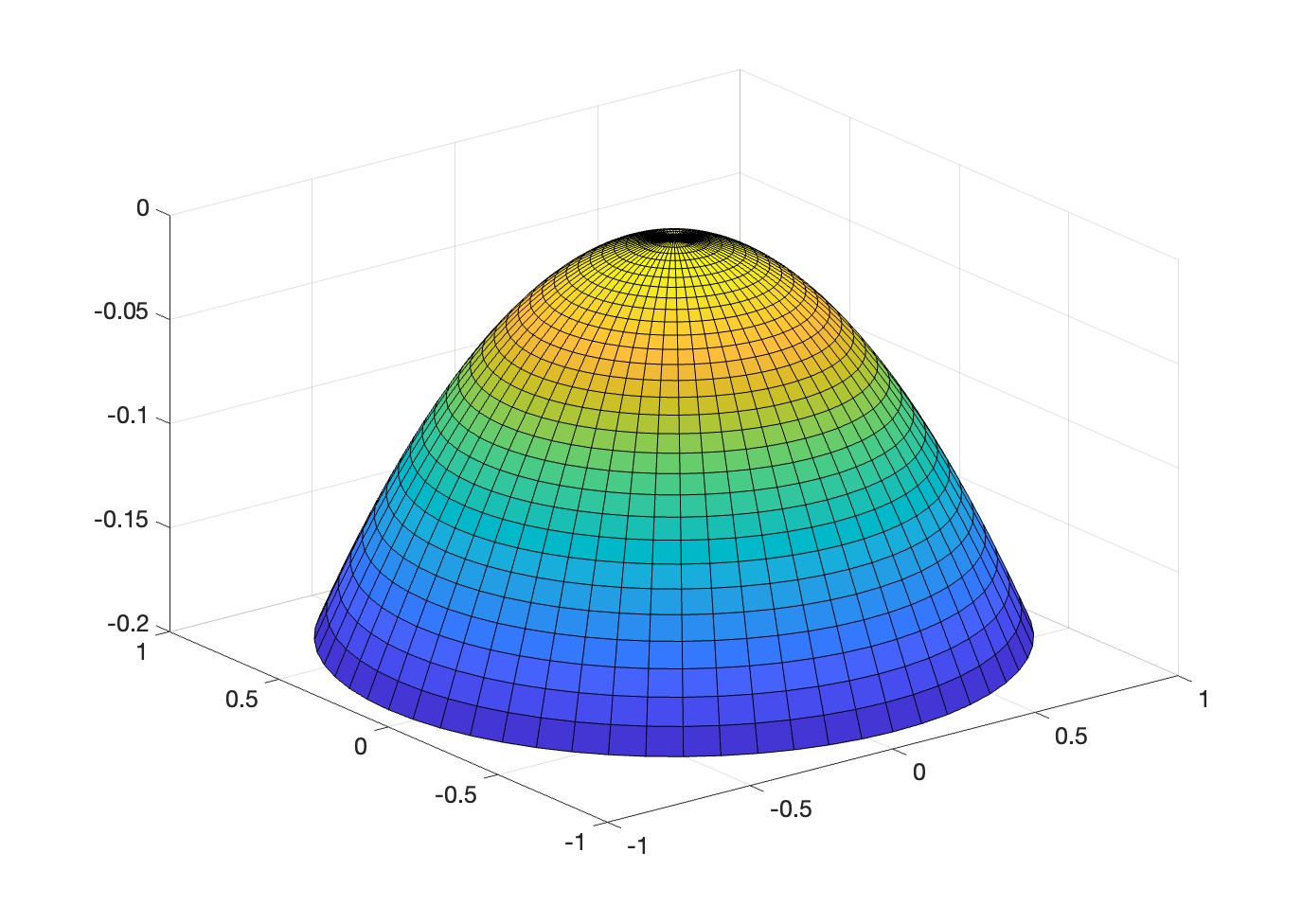}
			\end{minipage} &
			\begin{minipage}[h]{0.5\textwidth}
				\includegraphics[width=4cm]{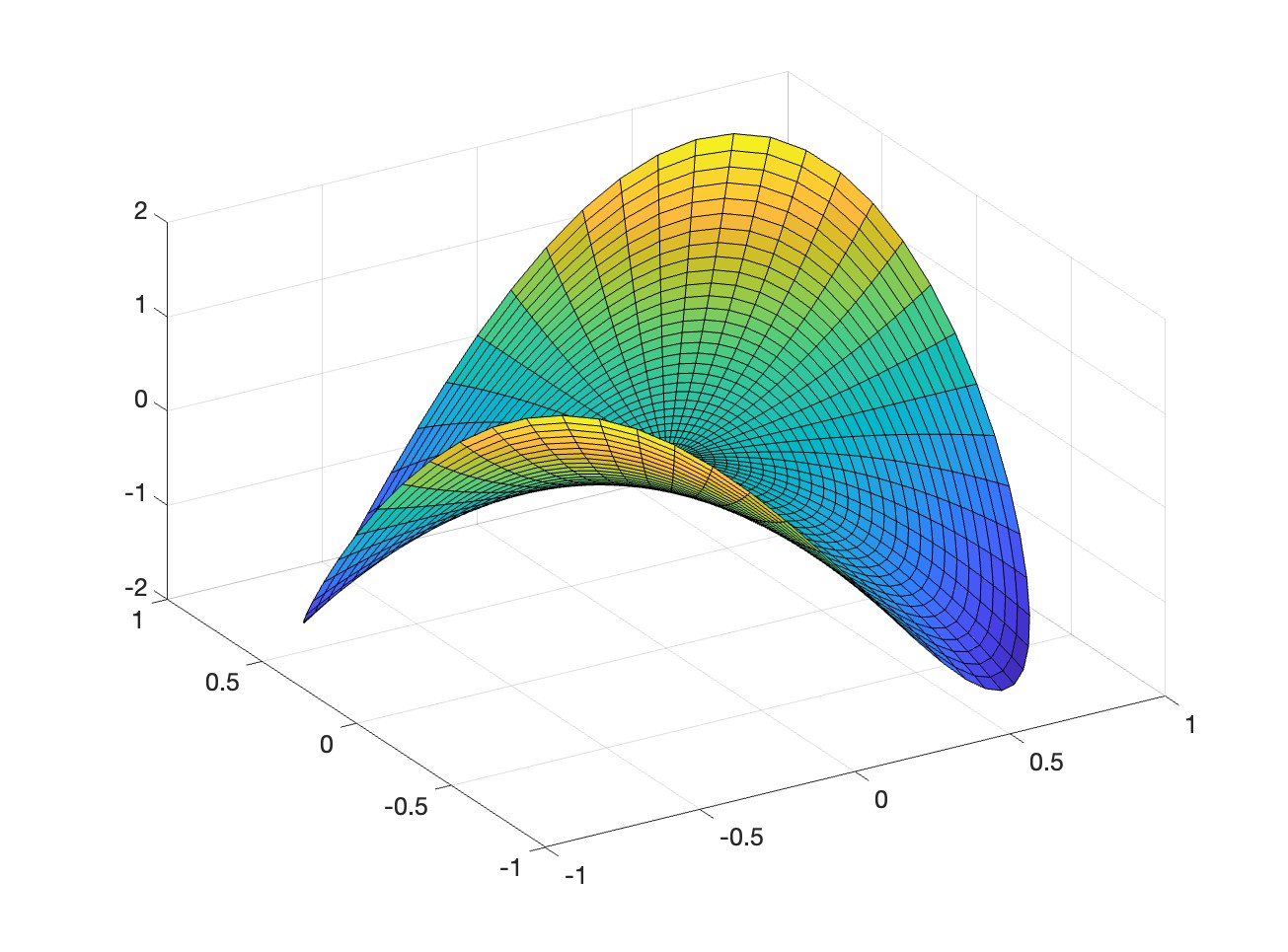}
			\end{minipage}\\
			\begin{minipage}[h]{0.45\textwidth}
				(a) Graph of Radial part of the normalized Clifford-Gegenbauer polynomial $C_{1,2}^{(-0.9)}(Y_1)$.
			\end{minipage} & 
			\begin{minipage}[h]{0.45\textwidth}
				(d) Graph of $e_{12}$ part of the normalized Clifford-Gegenbauer polynomial $C_{1,2}^{1}(Y_1)$.
			\end{minipage}\\
			\begin{minipage}[h]{0.5\textwidth}
				\includegraphics[width=4cm]{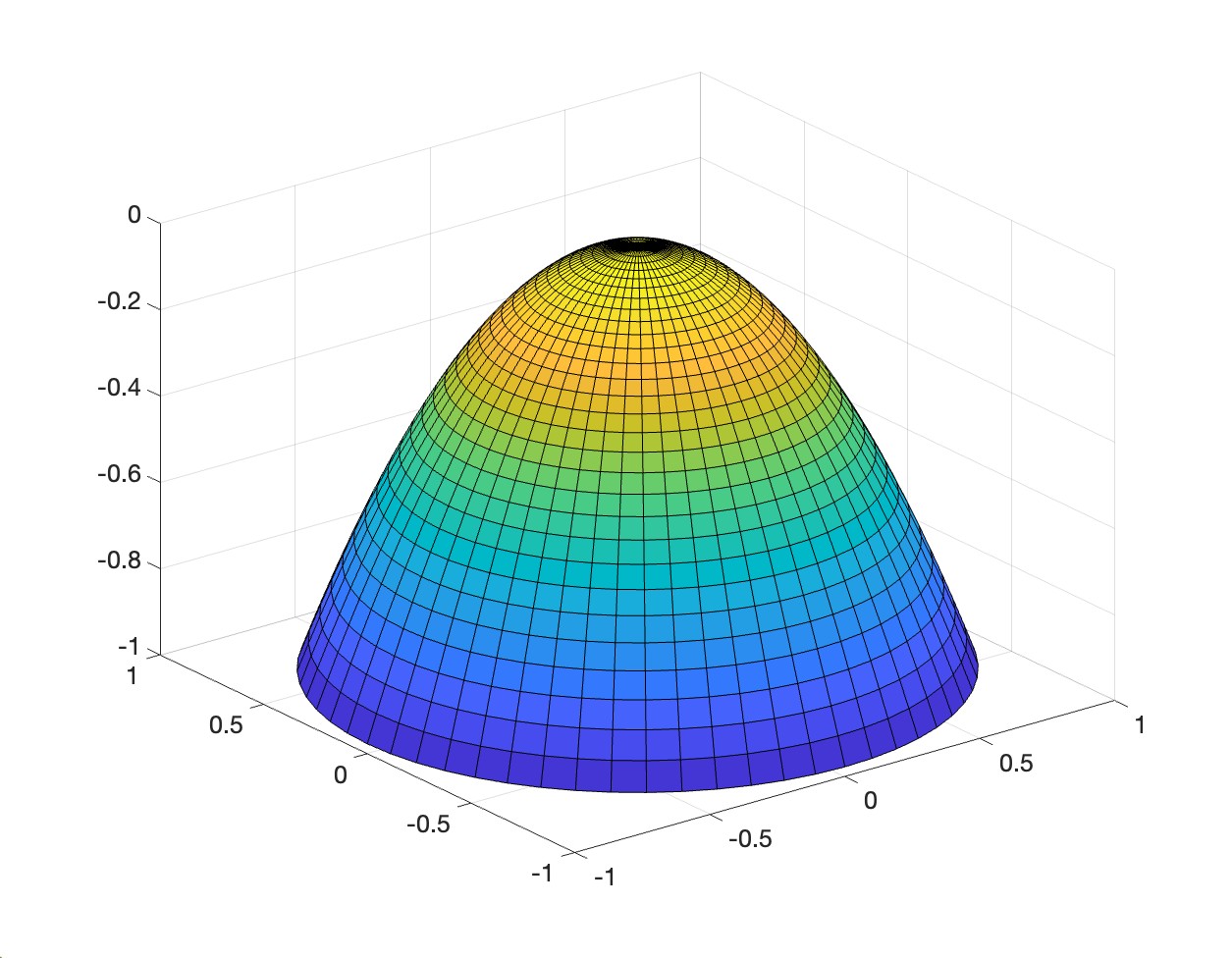}
			\end{minipage}&
			\begin{minipage}[h]{0.5\textwidth}
				\includegraphics[width=4cm]{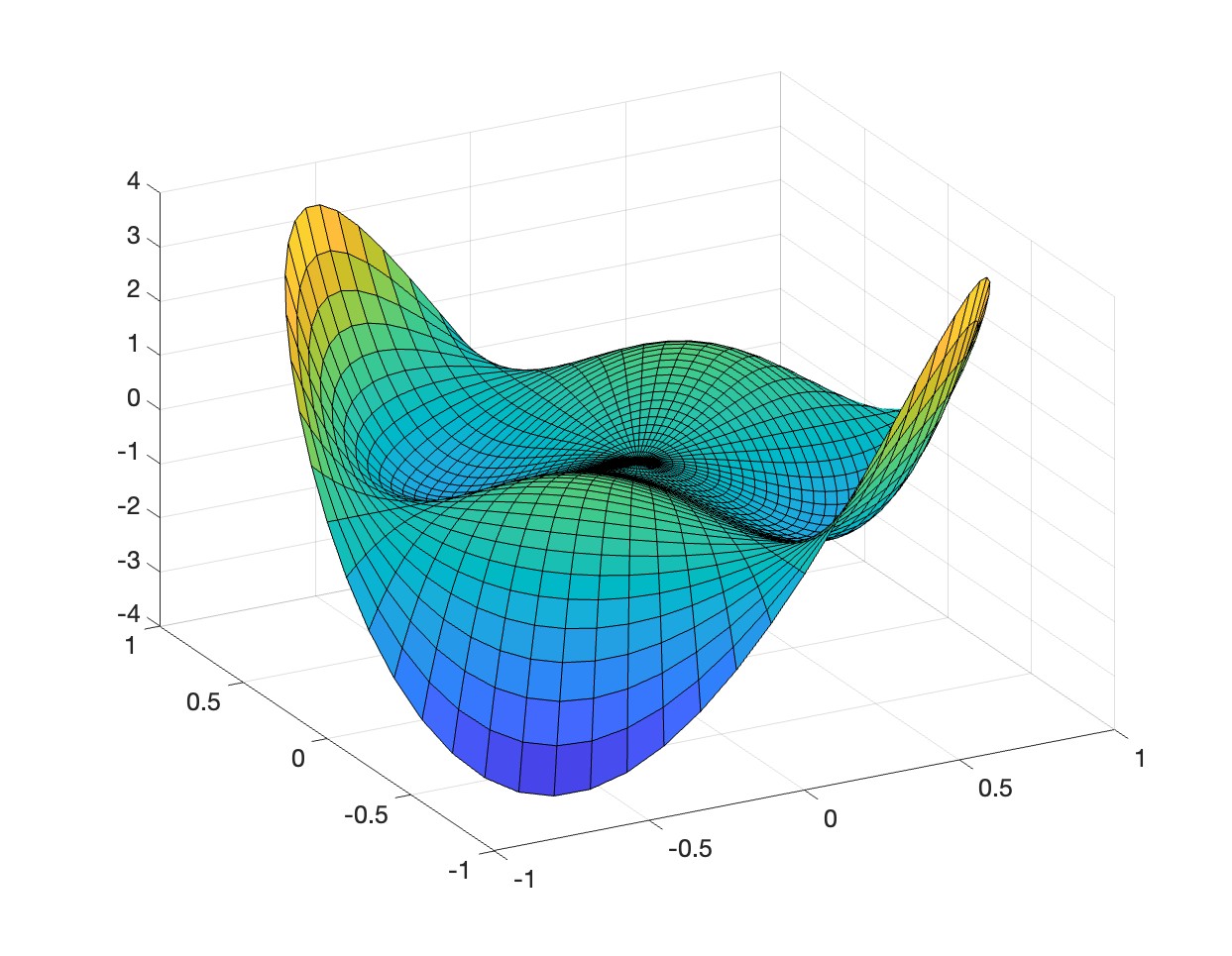}
			\end{minipage}\\
			\begin{minipage}[h]{0.45\textwidth}
				(b) Graph of Radial part of the normalized Clifford-Gegenbauer polynomial $C_{1,2}^0(Y_1)$.
			\end{minipage} &
			\begin{minipage}[h]{0.45\textwidth}
				(e) Graph of $e_{12}$ part of the normalized Clifford-Gegenbauer polynomial $C_{3,2}^{1}(Y_1)$.
			\end{minipage}\\		
			\begin{minipage}[h]{0.5\textwidth}
				\includegraphics[width=4cm]{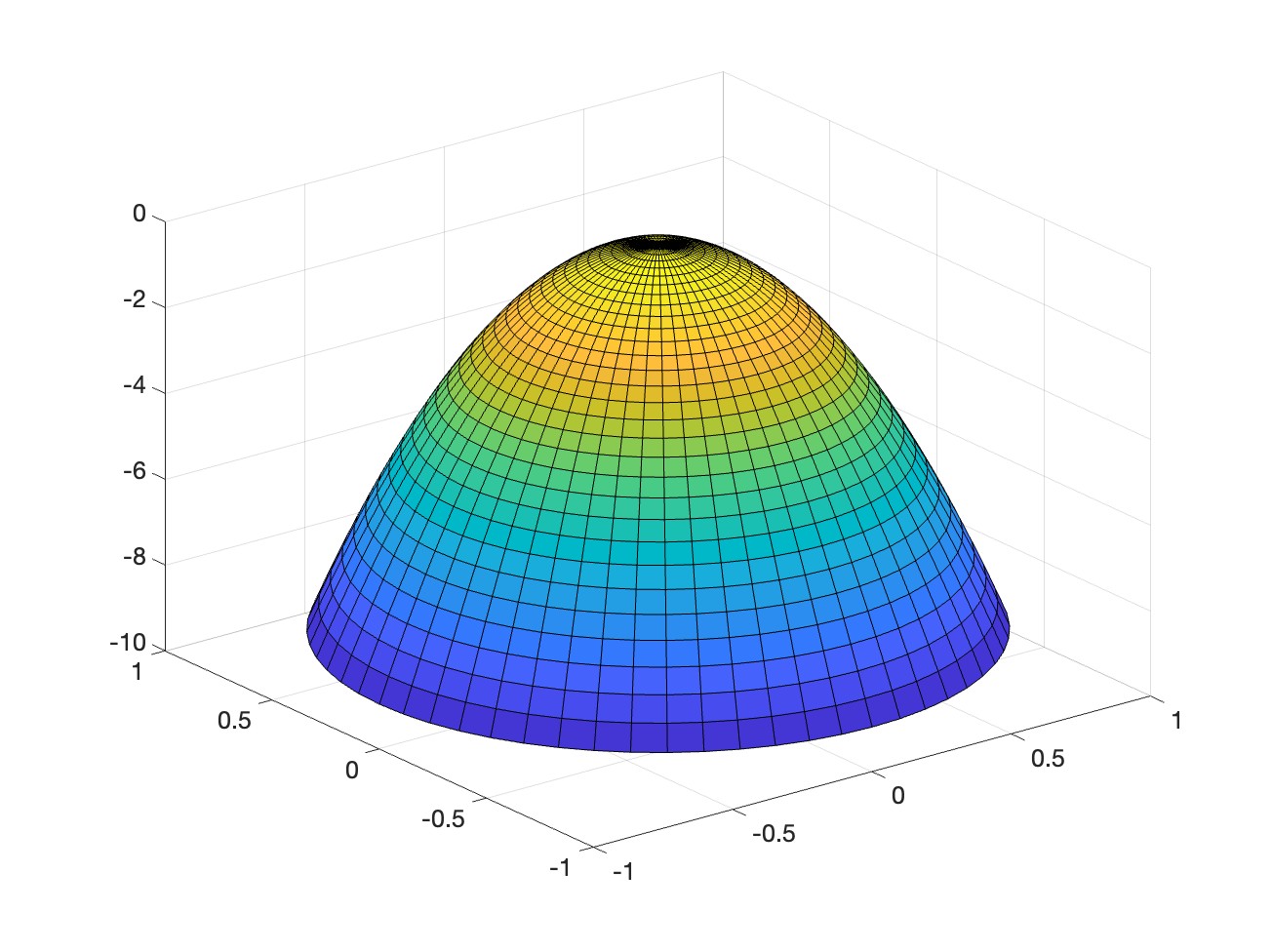}
			\end{minipage}&
			\begin{minipage}[h]{0.5\textwidth}
				\includegraphics[width=4cm]{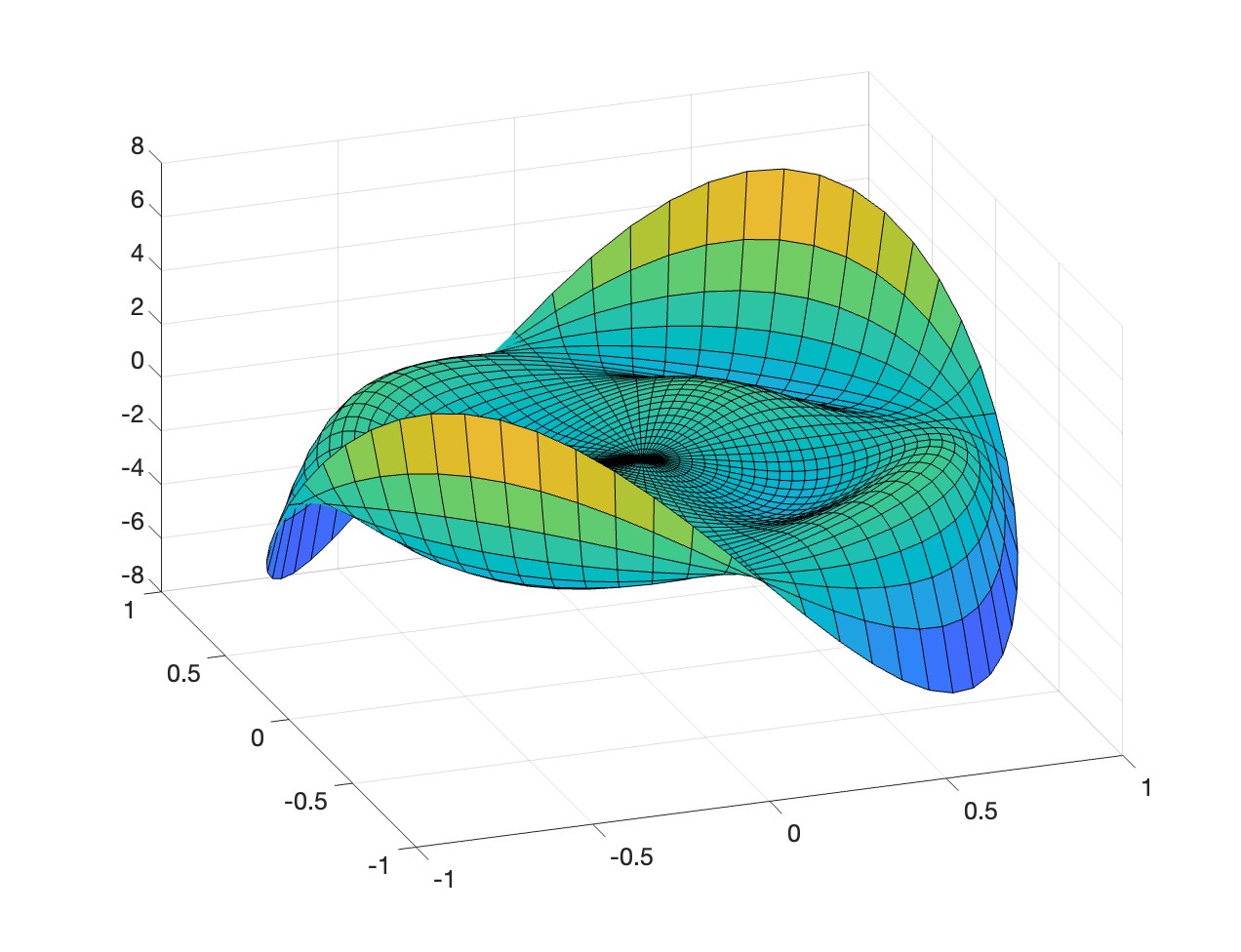}
			\end{minipage}\\
			\begin{minipage}[h]{0.45\textwidth}
				(c) Graph of Radial part of the normalized Clifford-Gegenbauer polynomial $C_{1,2}^6(Y_1)$.
			\end{minipage} &
			\begin{minipage}[h]{0.45\textwidth}
				(f) Graph of $e_{12}$ part of the normalized Clifford-Gegenbauer polynomial $C_{5,2}^{1}(Y_1)$.
			\end{minipage}
		\end{tabular}
		\caption{Plots of $C_{2N+1,2}^\alpha(Y_k)$ }
		\label{CGPs_Figure2}
	\end{figure}
	\section{Weighted CPSWFs}
	In this section, we define the weighted CPSWFs and then we will compute using Bonnet formulas.
	
	\begin{Definition}
		Let $c\geq 0$ and $\alpha>-1$. The Clifford operator $\mathcal{L}_{c,\alpha}$ acting on $C^{2}(B(1),\mathbb{R}_{m})\subset L_{2}(B(1),(1-\vert x\vert^2)^\alpha)$ is defined as follows
		\begin{equation}\label{L_c_alpha_operator}
			\mathcal{L}_{c,\alpha}f(x)=(1-\vert x\vert^{2})^{-\alpha}\partial_{x}((1-\vert x\vert^{2})^{\alpha+1}\partial_{x}f(x))+4\pi^{2}c^{2}\vert x\vert^{2}f(x).
		\end{equation}
	\end{Definition}
	\begin{Th}
		The operator $L_{c,\alpha}$ defined in \eqref{L_c_alpha_operator} is self-adjoint regarding to the weighted inner product, i.e., 
		\begin{equation}
			\langle f,\mathcal{L}_{c,\alpha}g \rangle_{\alpha}=\langle \mathcal{L}_{c,\alpha}f,g \rangle_{\alpha}.
		\end{equation}
	\end{Th}
	\begin{proof}
		We see that $\mathcal{L}_{c,\alpha}f=L_{0,\alpha}f+4\pi^{2}c^{2}\vert x\vert^{2}f(x)$. We prove that $\mathcal{L}_{\alpha}$ is self-adjoint. So using Clifford-Stokes theorem given in \ref{Clifford-Stokes theorem} we have that 
		\begin{align*}
			\langle f,\mathcal{L}_{0,\alpha}g \rangle_{\alpha}&=\int_{B(1)}\overline{f(x)}(L_{0,\alpha}g)(x)(1-\vert x\vert^{2})^{\alpha}\, dx\\
			&=\int_{B(1)}\overline{f(x)}\partial_{x}((1-\vert x\vert^{2})^{\alpha+1}\partial_{x}g(x)) \,dx\\
			&=\int_{\partial B(1)}(\overline{f(x)})\frac{x}{\vert x\vert}(1-\vert x\vert^{2})^{\alpha+1}\partial_{x}g(x) \,dx-\int_{B(1)}(\overline{f(x)}\partial_{x})((1-\vert x\vert^{2})^{\alpha+1})\partial_{x}f(x)\, dx\\
			&=\int_{B(1)}(\overline{f(x)}\partial_{x})((1-\vert x\vert^{2})^{\alpha+1})\partial_{x}f(x)\, dx\\
			&=\int_{B(1)}(1-\vert x\vert^{2})^{-\alpha}((\overline{f(x)}\partial_{x})((1-\vert x\vert^{2})^{\alpha+1})\partial_{x}f(x))(1-\vert x\vert^{2})^{\alpha}\, dx\\
			&=\langle \mathcal{L}_{0,\alpha}f,g \rangle_{\alpha}.
		\end{align*}
		This help to complete the proof.	
	\end{proof}
	\begin{Definition}
		The eigenfunctions of the operator $\mathcal{L}_{c,\alpha}$ defined in \eqref{L_c_alpha_operator} are called weighted CPSWFs, $\psi_{n,m,\alpha}^{k,c,i}(x)$.
	\end{Definition}
	Here we compute the weighted CPSWFs. With twice application of the normalized version of Bonnet formulas given in \eqref{normalized_even_Bonnet_formulae_Gegenbauer} and \eqref{normalized_odd_Bonnet_formulae_Gegenbauer}, we can see that
	\begin{equation*}
		\vert x\vert^{2}\barC_{2i,m}^{\alpha}(Y_{k}^{j})(x)=a_{i}\barC_{2i+2,m}^{\alpha}(Y_{k}^{j})(x)+b_{i}\barC_{2i,m}^{\alpha}(Y_{k}^{j})(x)+c_{i}\barC_{2i-2,m}^{\alpha}(Y_{k}^{j})(x),
	\end{equation*}
	in which $a_{i}=-\overline{A_{i,\alpha,k,m}}\, \overline{A'_{i,\alpha,k,m}},$ $b_{i}=-\big[ \,\overline{A_{i,\alpha,k,m}}\, \overline{B'_{i,\alpha,k,m}} + \overline{B_{i,\alpha,k,m}}\, \overline{A'_{i-1,\alpha,k,m}}\, \big],$ and, $c_{i}=-\overline{B_{i,\alpha,k,m}}\, \overline{B'_{i-1,\alpha,k,m}}$. The weighted CPSWFs, $\psi_{n,m,\alpha}^{k,c,i}(x)$, belong to the weighted space $L_{2}(B(1),(1-\vert x\vert^2)^\alpha)$. By direct computation, it may be shown that if $X_k^e$ is the collection of functions of the form $F(|x|^2)Y_{k}^j(x)$ with $F\in C^\infty ([0,1],{\mathbb R})$ and $Y_{k}$ is a fixed spherical monogenic of degree $k$ on ${\mathbb R}^m$, then $X_k^e$ is invariant under $L_c$. In fact, a direct calculation, using the monogenicity and homogeneity of $Y_{k}$ gives
	$$\mathcal{L}_{0}[F(|x|^2)Y_{k}^{j}(x)]=[(-2m+(4+4k)|x|^2)F'(|x|^2)-4k|x|^2F''(|x|^2)]Y_{k}^{j}(x).$$
	Similarly, if $X_k^o$ is the collection of functions of the form $xG(|x|^2)Y_{k}$ with $G\in C^\infty ([0,1],{\mathbb R})$ and $Y_{k}$ a spherical monogenic of degree $k$ on ${\mathbb R}^m$, then $X_k^0$ is also invariant under $\mathcal{L}_{c,\alpha}$. Hence, when searching for eigenfunctions of $\mathcal{L}_{c,\alpha}$, we can search within the spaces $X_k^e$ and $X_k^o$. Furthermore, the spaces $X_k^e$ and $X_\ell^o$ are orthogonal for all values of $k$ and $\ell$. The collection $\{\barC_{2i,m}^{\alpha}(Y_{k}^{j})\}_{i=0}^\infty$ lies in $X_k^e$ while $\{\barC_{2i+1,m}^{\alpha}(Y_{k}^{j})\}_{i=0}^\infty$ lies in $X_k^e$. We can therefore assume that eigenfunctions of $\mathcal{L}_{c,\alpha}$ will take one of two forms. Therefore, as in \cite{ghaffari2021clifford}, we have that
	
	\begin{align*}
		\mathcal{L}_{c,\alpha}\psi_{2N,m,\alpha}^{k,c,i}(x)&=\sum_{i,k'=0}^{\infty}\big[ \barC_{2i,m}^{\alpha}(Y_{k'}^{j})(x) \big] \, \alpha_{i,N,m,\alpha}^{k}=\sum_{i=0}^{\infty}\big[ \barC_{2i,m}^{\alpha}(Y_{k}^{j})(x) \big]\,\alpha_{i,N,m,\alpha}^{k}\\
		&=\sum_{i=0}^{\infty} \barC_{2i,m}^{\alpha}(Y_{k}^{j})(x) \big[(4\pi^{2}c^{2}a_{i-1})\alpha_{i-1,N,m,\alpha}^{k} + (C(\alpha,2i,m,k)+4\pi^{2}c^{2}b_{i})\alpha_{i,N,m,\alpha}^{k}\\ &\hspace*{5.15cm}+ (4\pi^{2}c^{2}c_{i+1})\alpha_{i+1,N,m,\alpha}^{k} \big]\\
		&=\chi_{2N,m,\alpha}^{k,c}\psi_{2N,m,\alpha}^{k,c,i}(x).
	\end{align*}
	with the understanding that 
	$\alpha_{i,N,m,\alpha}^{k}=0$
	if
	$i<0$.
	The orthogonality of the CGPs enables us to equate the coefficients on both sides of this equation to obtain the following recurrence formula for $\{\alpha_{i,N,m,\alpha}^{k}\}_{i=0}^{\infty}$:
	$$(4\pi^2c^2 a_{i-1})\alpha_{i-1,N,m,\alpha}+\big(C(\alpha,2i,m,k)+4\pi^2c^2b_{i}-\chi_{2N,m,\alpha}^{k,c}\big)\alpha_{i,N,m,\alpha}+(4\pi^2c^2c_{i+1})\alpha_{i+1,N,m,\alpha}^{k}=0.$$
	This recurrence formula holds for all $i,N\geq 0$ so the problem reduces to finding the eigenvectors $\alpha_{i,N,m}^{k}$ and associated eigenvalues $\chi_{2N,m,\alpha}^{k,c}$ of the doubly-infinite matrix $M_{k,m}^{e}$ with the following entries:
	\begin{eqnarray*}
		M_{k,m,\alpha}^{e}(i,j)=\left\lbrace \begin{array}{ll}
			\vspace{.5cm}
			-4\pi^{2}c^{2}\frac{\sqrt{(\alpha+i)i(k+\frac{m}{2}+i-1)(\alpha+k+\frac{m}{2}+i-1)}}{(\alpha+k+\frac{m}{2}+2i-1)\sqrt{(\alpha+k+\frac{m}{2}+2i)(\alpha+k+\frac{m}{2}+2i-2)}},\;\;\;\textnormal{if}\; i\geq 1,\; j=i-1,  \\
			\vspace{.5cm} 4i(\alpha+k+\frac{m}{2}+i)+\frac{4\pi^2c^2}{(\alpha+k+\frac{m}{2}+2i)}\big[\frac{(\alpha+k+\frac{m}{2}+i)(k+\frac{m}{2}+i)}{(\alpha+k+\frac{m}{2}+2i+1)}+\frac{i(\alpha+i)}{(\alpha+k+\frac{m}{2}+2i-1)}\big],\;\;\;\textnormal{if}\; i=j\geq 0,\\
			\vspace{.5cm}
			-4\pi^{2}c^{2}\frac{\sqrt{(\alpha+i+1)(i+1)(k+\frac{m}{2}+i)(\alpha+k+\frac{m}{2}+i)}}{(\alpha+k+\frac{m}{2}+2i+1)\sqrt{(\alpha+k+\frac{m}{2}+2i+2)(\alpha+k+\frac{m}{2}+2i)}},\;\;\;\textnormal{if}\; i\geq 0,\; j=i+1,  \\
			0,\;\;\;\;\;\;\;\; else.
		\end{array} \right.
	\end{eqnarray*}
	We see directly that $M_{k,m,\alpha}^{e}$ is symmetric and banded in the sense that $M_{k,m}^{e}(i,j)=0$ if $|i-j|>1$. The symmetry of $M_{k,m,\alpha}^{e}$ is due to the self-adjointness of $\mathcal{L}_{c,\alpha}$.
	
	Similarly, we can obtain the matrix coefficient for the odd version
	\begin{eqnarray*}
		M_{k,m,\alpha}^{o}(i,j)=\left\lbrace \begin{array}{ll}
			\vspace{.5cm}
			-4\pi^{2}c^{2}\frac{\sqrt{(\alpha+i)i(k+\frac{m}{2}+i)(\alpha+k+\frac{m}{2}+i)}}{(\alpha+k+\frac{m}{2}+2i)\sqrt{(\alpha+k+\frac{m}{2}+2i-1)(\alpha+k+\frac{m}{2}+2i+1)}},\;\;\;\;\;\;\textnormal{if}\; i\geq 1,\; j=i-1,  \\
			\vspace{.2cm}
			4(\alpha+i+1)(k+\frac{m}{2}+i)+\frac{4\pi^2c^2}{(\alpha+k+\frac{m}{2}+2i+1)}\big[\frac{(\alpha+k+\frac{m}{2}+i)(k+\frac{m}{2}+i)}{(\alpha+k+\frac{m}{2}+2i)}+\frac{(i+1)(\alpha+i+1)}{(\alpha+k+\frac{m}{2}+2i+2)}\big],\\
			\vspace{.5cm}\hspace*{12cm}\;\textnormal{if}\; i=j\geq 0,\\
			\vspace{.5cm}
			-4\pi^{2}c^{2}\frac{\sqrt{(\alpha+i+1)(i+1)(k+\frac{m}{2}+i+1)(\alpha+k+\frac{m}{2}+i+1)}}{(\alpha+k+\frac{m}{2}+2i+2)\sqrt{(\alpha+k+\frac{m}{2}+2i+1)(\alpha+k+\frac{m}{2}+2i+3)}},\;\;\;\textnormal{if}\; i\geq 0,\; j=i+1,  \\
			0,\;\;\;\;\;\;\;\; else.
		\end{array} \right.
	\end{eqnarray*}
	We can see some plots for different parameters in figures \ref{Weighted_CPSWFs_Figure1},\ref{Weighted_CPSWFs_Figure2},\ref{Weighted_CPSWFs_Figure3}, and, \ref{Weighted_CPSWFs_Figure}.
	\begin{Remark}\label{Relationship_Coefficients_CPSWFs}
		We can conclude that the even $M_{k,m,\alpha}^{e}$ and odd $M_{k,m,\alpha}^{o}$ matrices in  any dimension $m$ have the following properties
		\begin{align*}
			&M_{k+1,m,\alpha}^{e}(i-1,i)=M_{k+1,m,\alpha}^{e}(i,i+1)=M_{k,m,\alpha}^{o}(i-1,i)=M_{k,m,\alpha}^{o}(i,i+1),\\
			&M_{k,m,\alpha}^{o}(i,i)=M_{k+1,m,\alpha}^{e}(i,i)+bI
		\end{align*}
		where $b=4(\alpha+1)(k+\frac{m}{2})$. Therefore,  if 
		$v$
		is the eigenvector of matrix of $M_{k,m,\alpha}^o$ with eigenvalue
		$\lambda$
		then
		$v$
		is also an eigenvector of $M_{k+1,m,\alpha}^e$ with eigenvalue
		$\lambda -b$.
		Now as in remark \ref{Relation_even_odd_CLGs_dimension_2} we can see that
		$$ \psi_{2N+1,2,\alpha}^{k,c}(Y_{k})(x)=-e_{1}\psi_{2N,2,\alpha}^{k+1,c}(Y_{k+1})(x), $$
		This is visible in figure \ref{Weighted_CPSWFs_Figure}.
	\end{Remark}
	
	\begin{figure}
		\centering
		\begin{tabular}{cc}
			\begin{minipage}[h]{0.5\textwidth}
				\includegraphics[width=4cm]{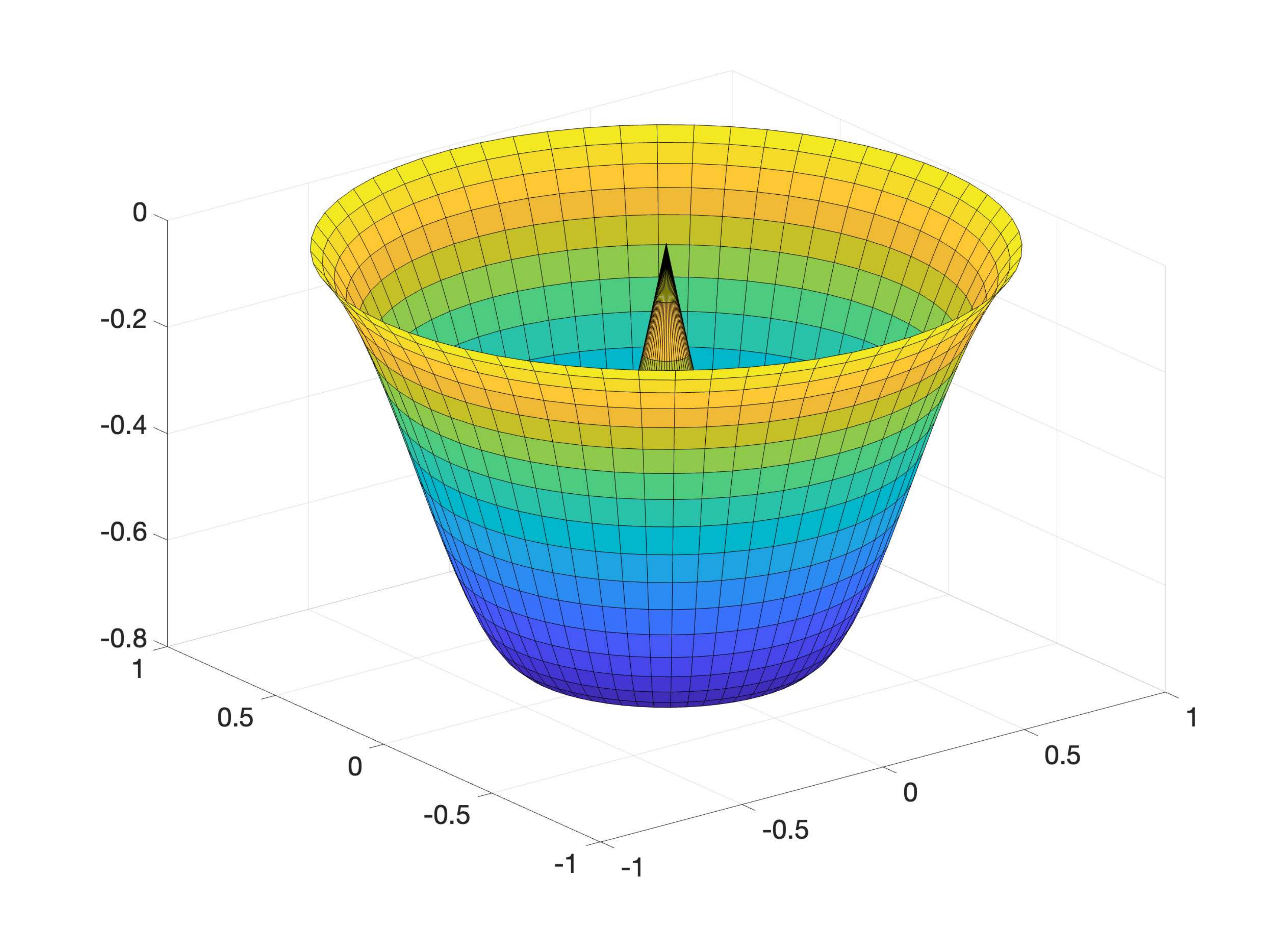}
			\end{minipage} &
			\begin{minipage}[h]{0.5\textwidth}
				\includegraphics[width=4cm]{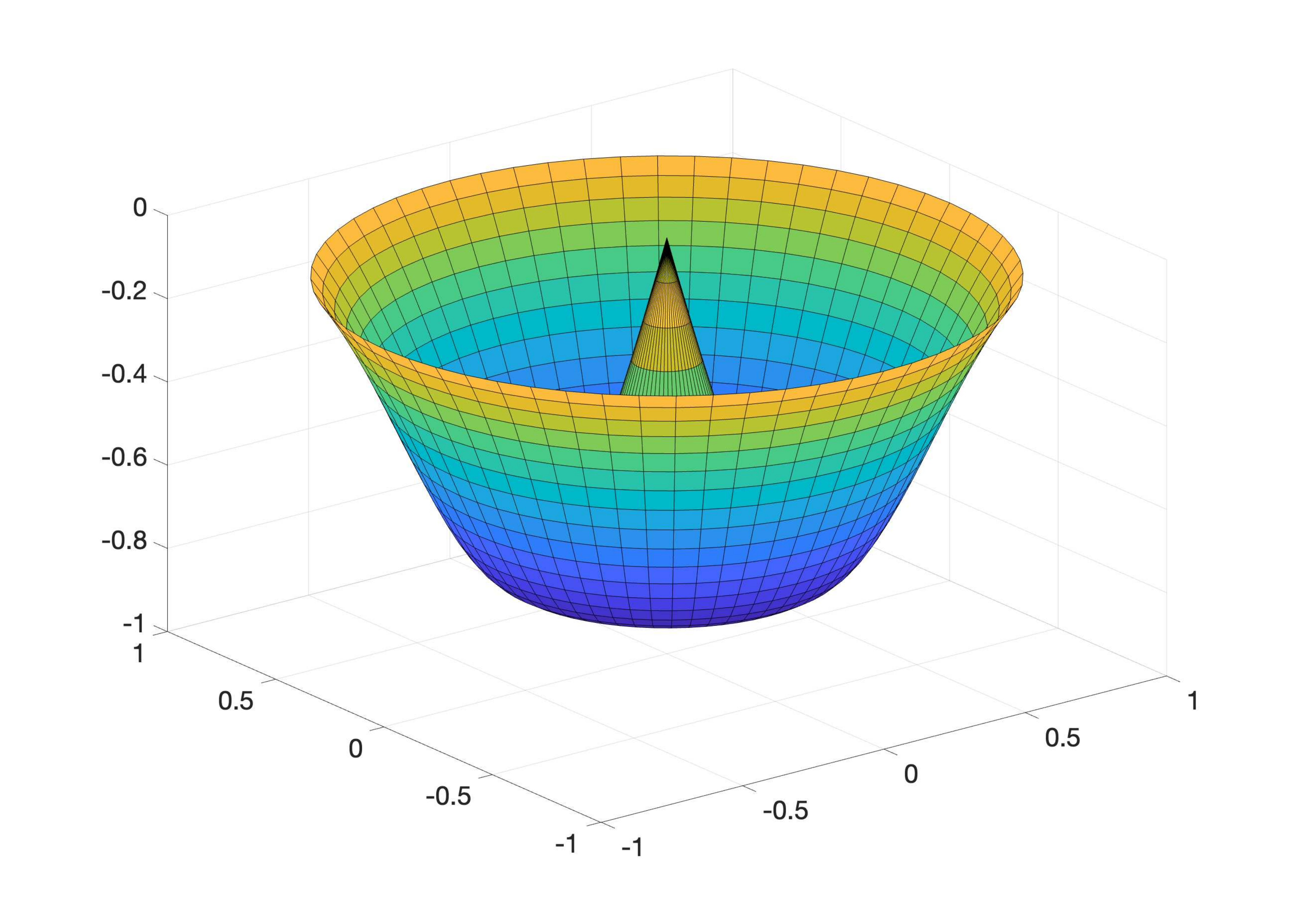}
			\end{minipage}\\
			\begin{minipage}[h]{0.45\textwidth}
				(a) Graph of Radial part of the normalized Weighted CPSWFs $\psi_{0,2,(-0.9)}^{1,1}$.
			\end{minipage} & 
			\begin{minipage}[h]{0.45\textwidth}
				(c) Graph of Radial part of the normalized Weighted CPSWFs $\psi_{0,2,(-0.2)}^{1,1}$.
			\end{minipage}\\
			\begin{minipage}[h]{0.5\textwidth}
				\includegraphics[width=4cm]{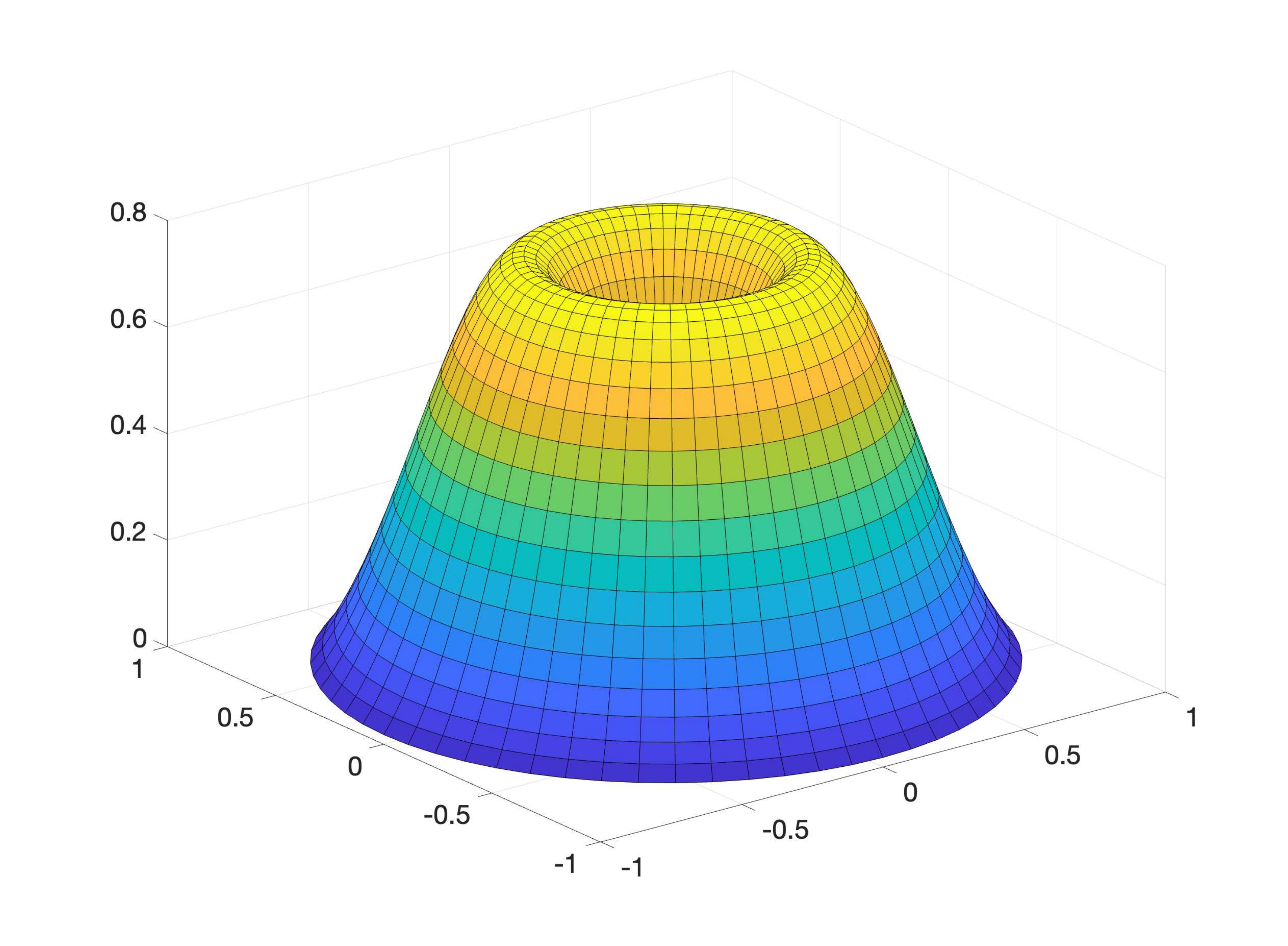}
			\end{minipage}&
			\begin{minipage}[h]{0.5\textwidth}
				\includegraphics[width=4cm]{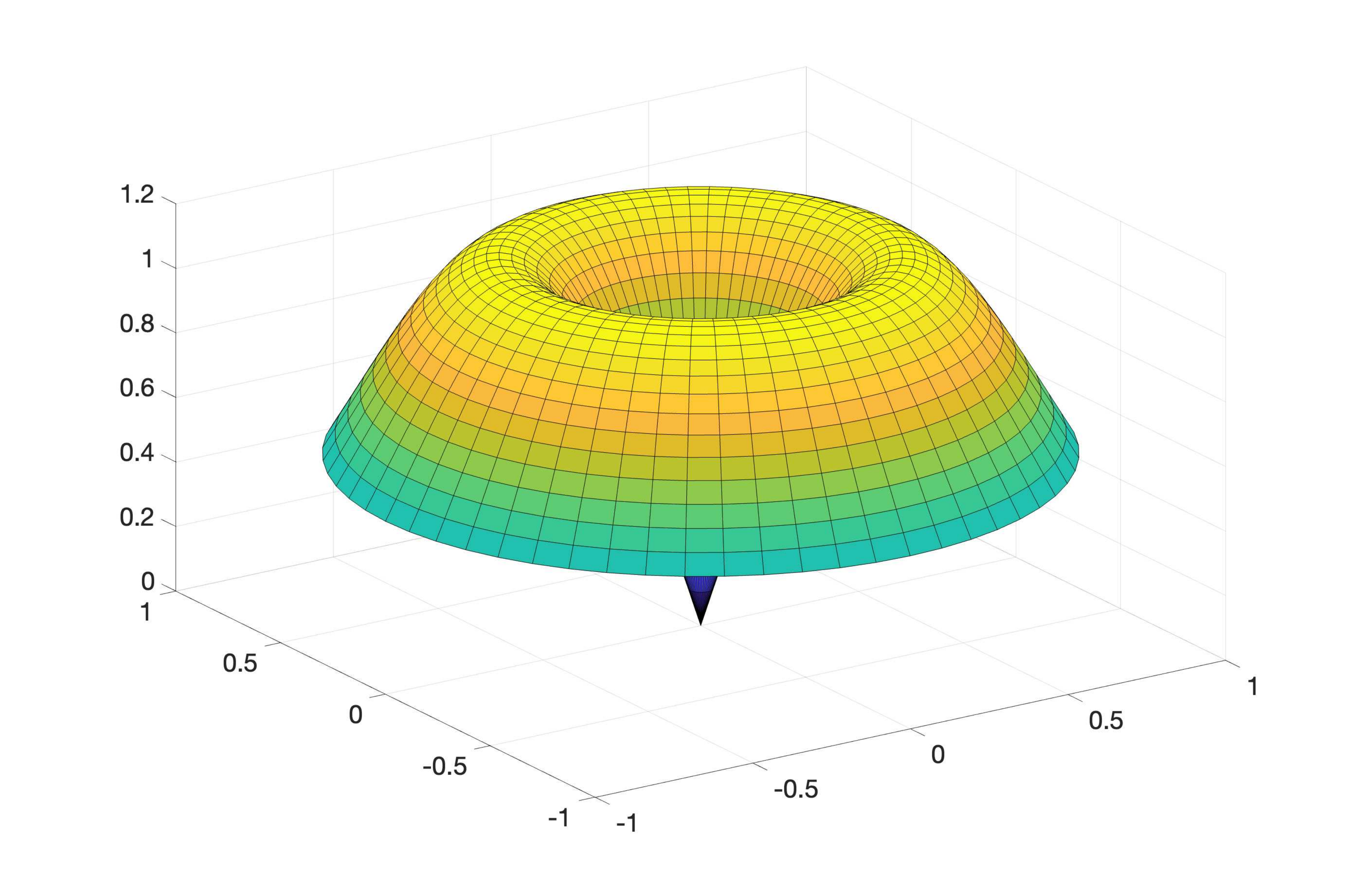}
			\end{minipage}\\
			\begin{minipage}[h]{0.45\textwidth}
				(b) Graph of Radial part of the normalized Weighted CPSWFs $\psi_{0,2,(-0.8)}^{1,1}$.
			\end{minipage} &
			\begin{minipage}[h]{0.45\textwidth}
				(d) Graph of Radial part of the normalized Weighted CPSWFs $\psi_{0,2,2}^{1,1}$.
			\end{minipage}
		\end{tabular}
		\caption{Plots of Radial part of $\psi_{0,2,\alpha}^{1,1}$ }
		\label{Weighted_CPSWFs_Figure1}
	\end{figure}
	
	\begin{figure}
		\centering
		\begin{tabular}{cc}
			\begin{minipage}[h]{0.5\textwidth}
				\includegraphics[width=4cm]{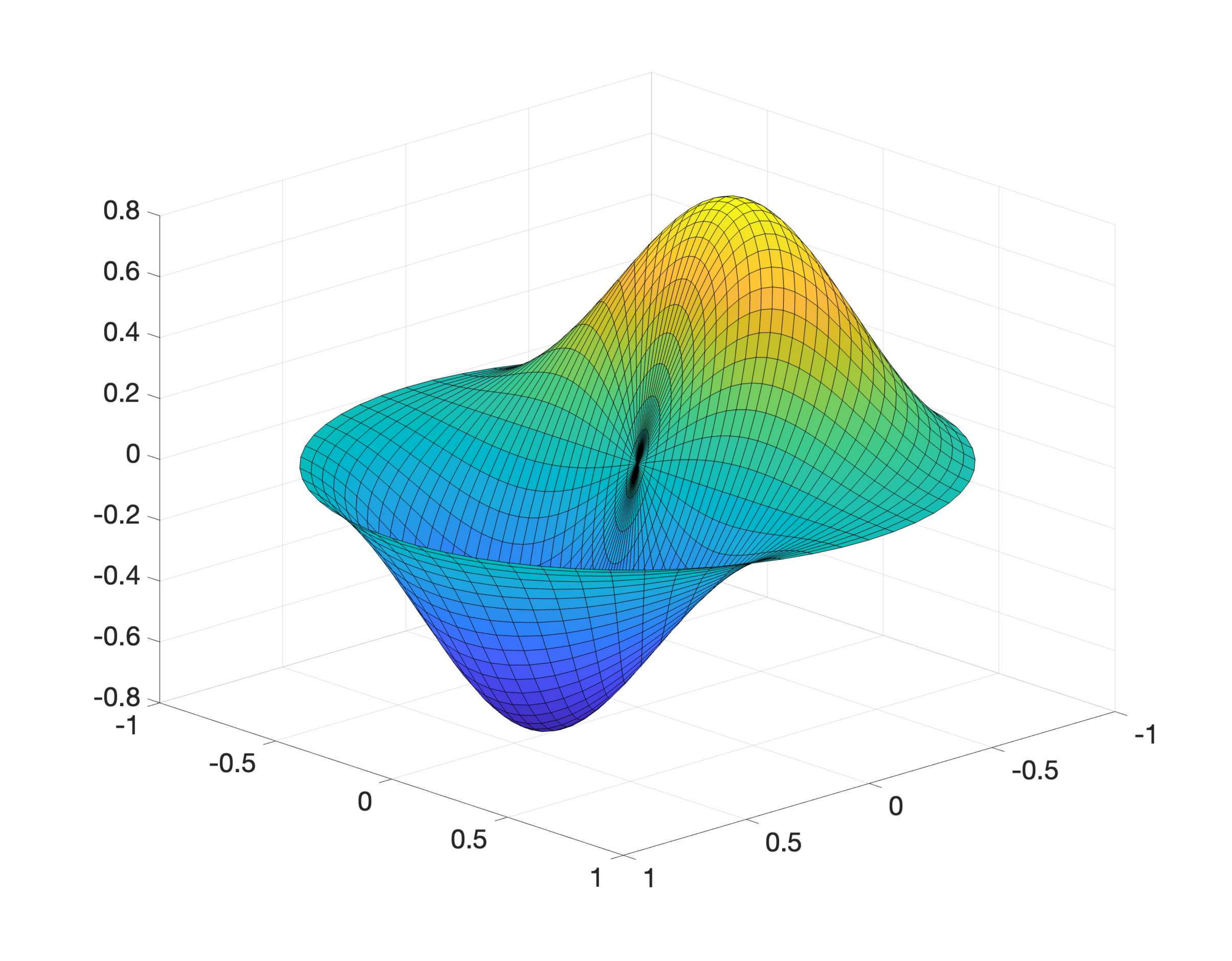}
			\end{minipage} &
			\begin{minipage}[h]{0.5\textwidth}
				\includegraphics[width=4cm]{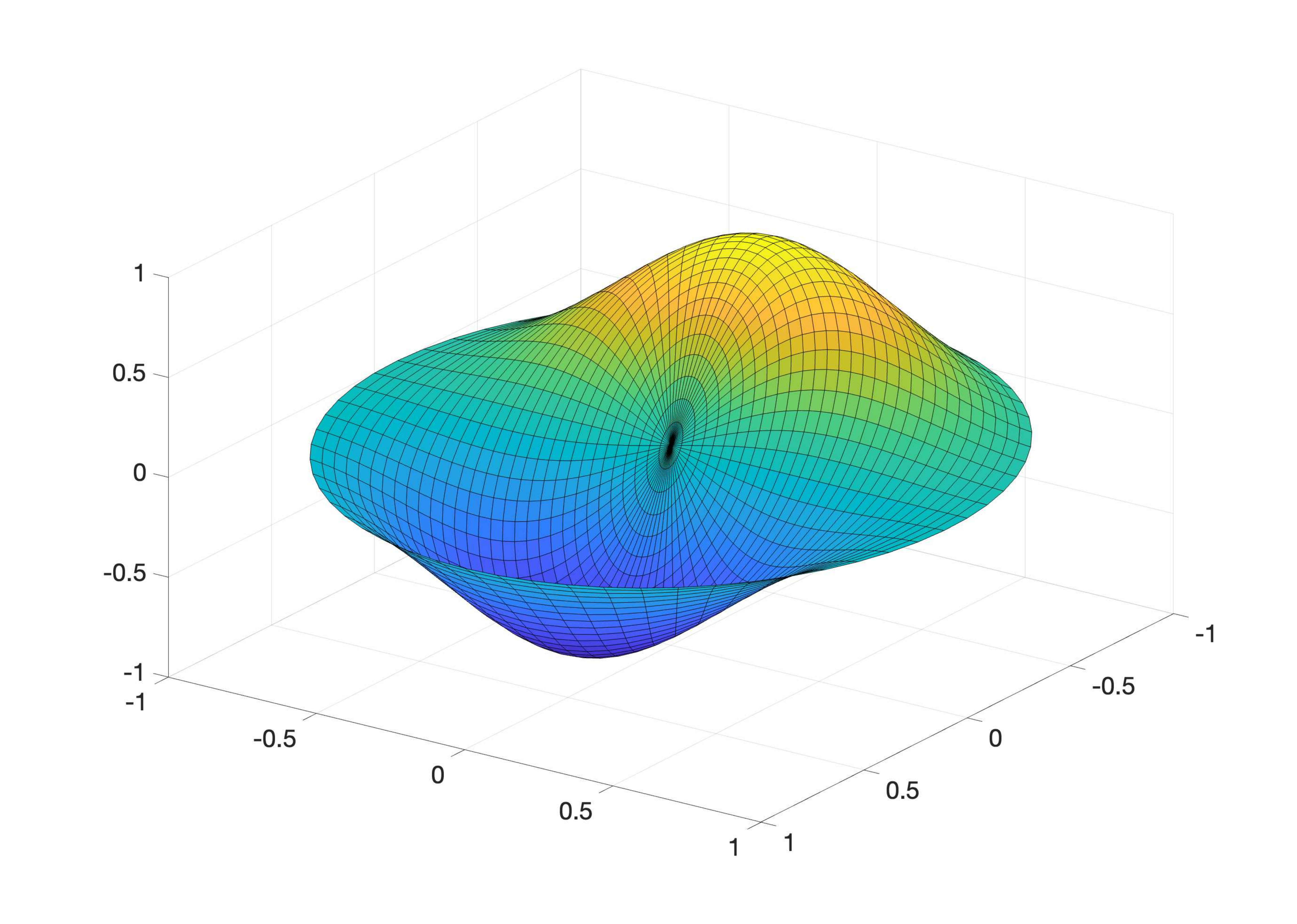}
			\end{minipage}\\
			\begin{minipage}[h]{0.45\textwidth}
				(a) Graph of $e_{1}$ part of the normalized Weighted CPSWFs $\psi_{0,2,(-0.9)}^{1,1}$.
			\end{minipage} & 
			\begin{minipage}[h]{0.45\textwidth}
				(c) Graph of $e_{1}$ part of the normalized Weighted CPSWFs $\psi_{0,2,(-0.2)}^{1,1}$.
			\end{minipage}\\
			\begin{minipage}[h]{0.5\textwidth}
				\includegraphics[width=4cm]{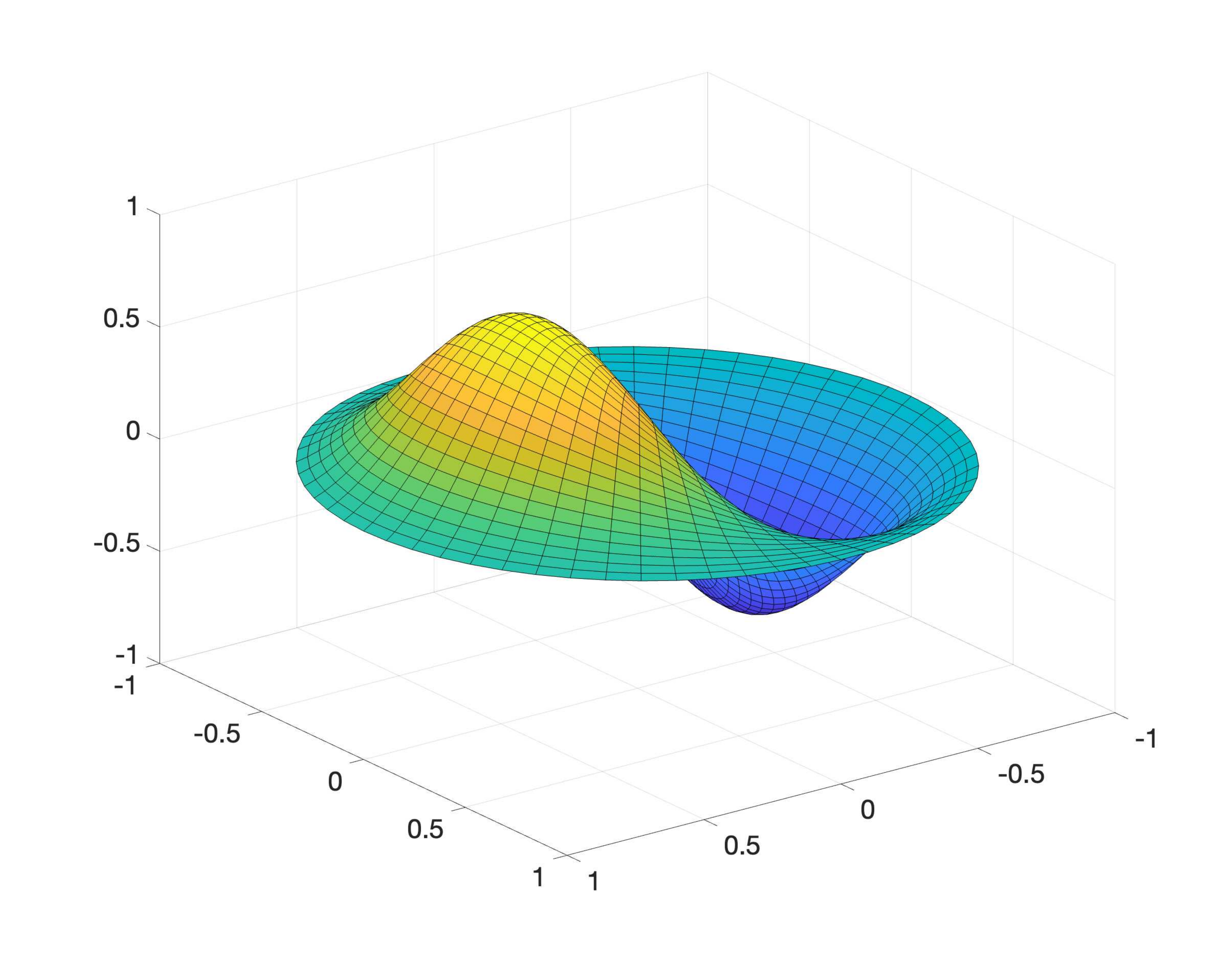}
			\end{minipage}&
			\begin{minipage}[h]{0.5\textwidth}
				\includegraphics[width=4cm]{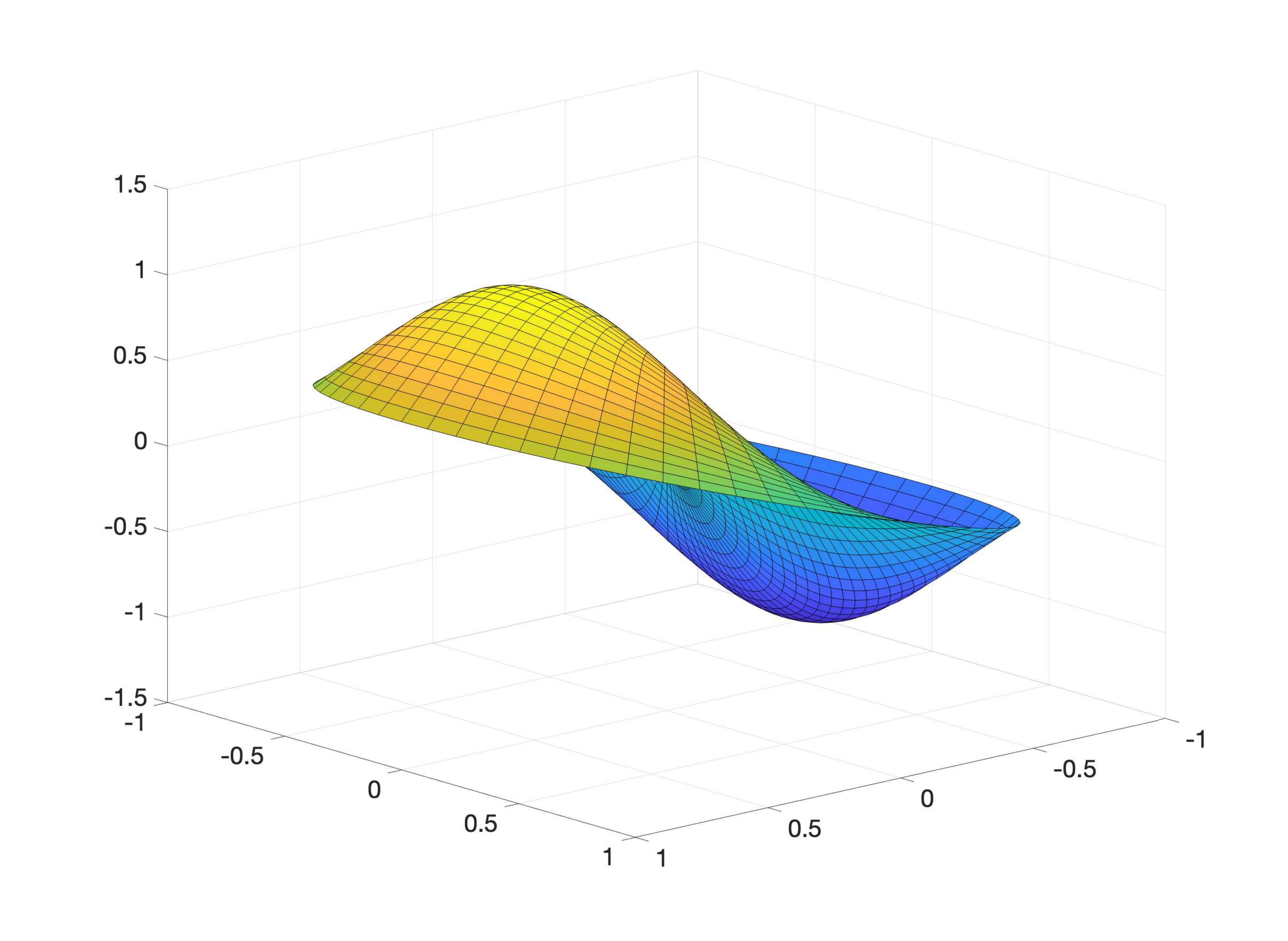}
			\end{minipage}\\
			\begin{minipage}[h]{0.45\textwidth}
				(b) Graph of $e_{1}$ part of the normalized Weighted CPSWFs $\psi_{0,2,(-0.8)}^{1,1}$.
			\end{minipage} &
			\begin{minipage}[h]{0.45\textwidth}
				(d) Graph of $e_{1}$ part of the normalized Weighted CPSWFs $\psi_{0,2,2}^{1,1}$.
			\end{minipage}
		\end{tabular}
		\caption{Plots of $e_{1}$  part of $\psi_{0,2,\alpha}^{1,1}$  }
		\label{Weighted_CPSWFs_Figure2}
	\end{figure}
	
	\begin{figure}
		\centering
		\begin{tabular}{cc}
			\begin{minipage}[h]{0.5\textwidth}
				\includegraphics[width=4cm]{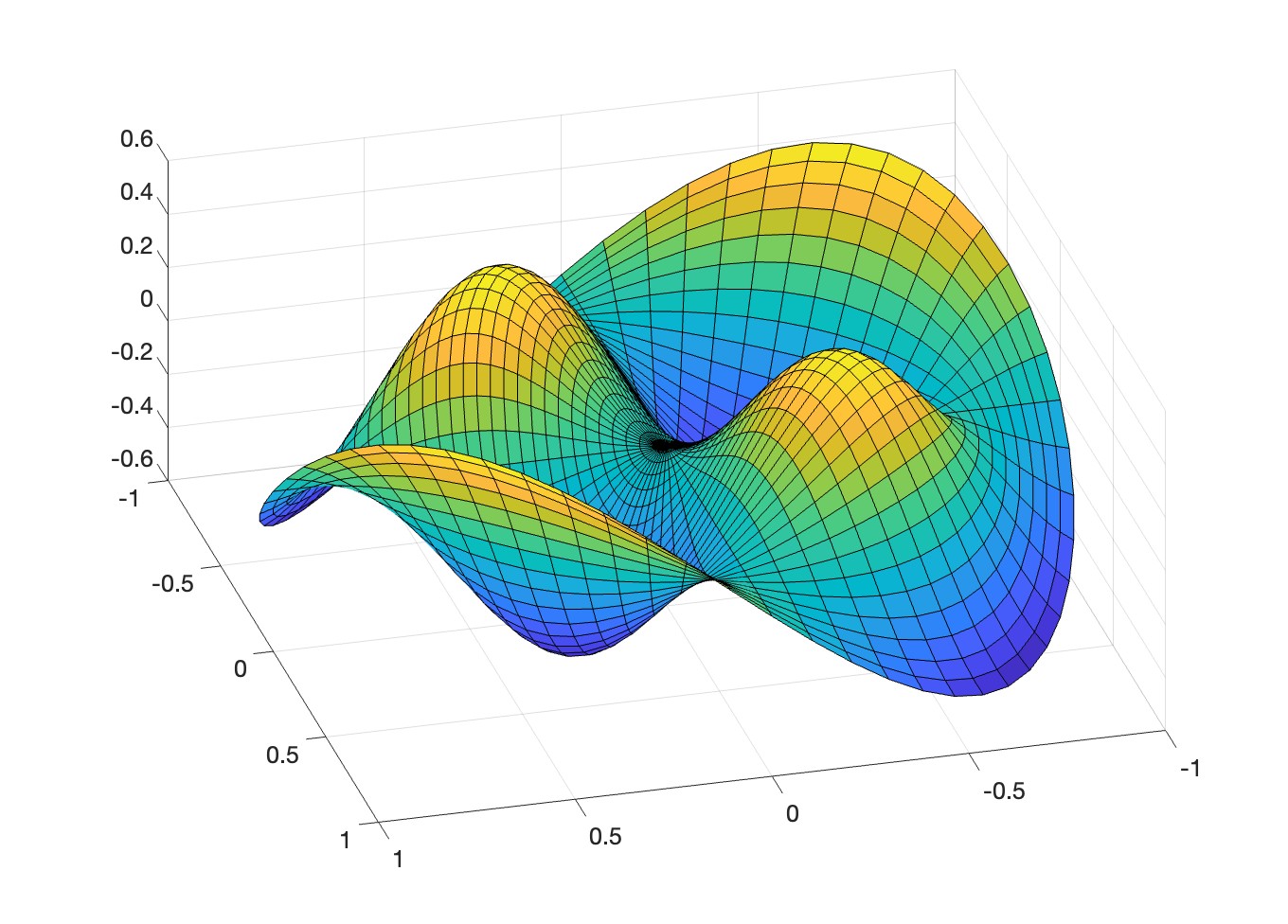}
			\end{minipage} &
			\begin{minipage}[h]{0.5\textwidth}
				\includegraphics[width=4cm]{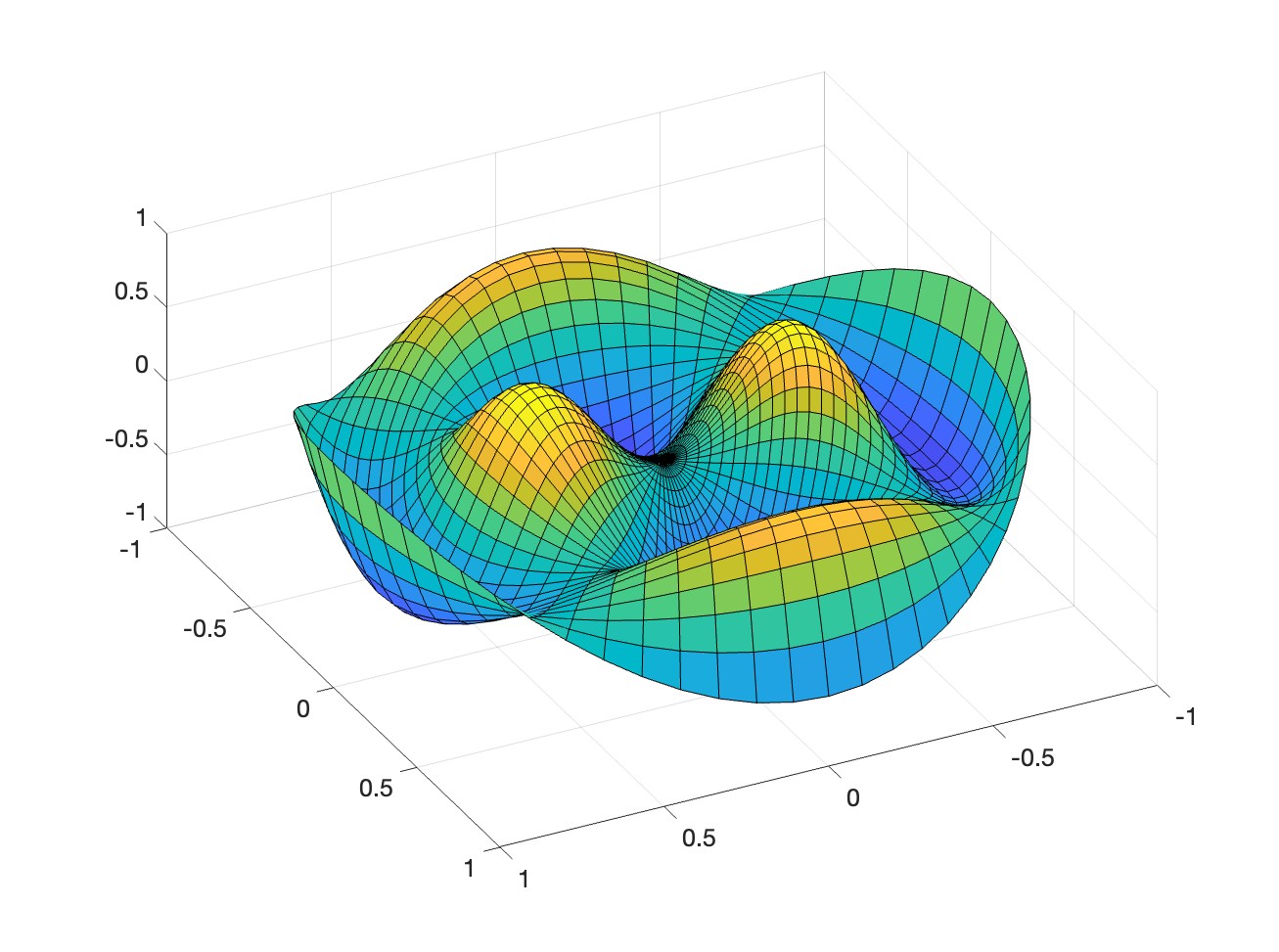}
			\end{minipage}\\
			\begin{minipage}[h]{0.45\textwidth}
				(a) Graph of $e_{12}$ part of the normalized Weighted CPSWFs $\psi_{3,2,(-0.5)}^{1,1}$.
			\end{minipage} & 
			\begin{minipage}[h]{0.45\textwidth}
				(c) Graph of $e_{0}$ part of the normalized Weighted CPSWFs $\psi_{5,2,(-0.5)}^{1,1}$.
			\end{minipage}\\
			\begin{minipage}[h]{0.5\textwidth}
				\includegraphics[width=4cm]{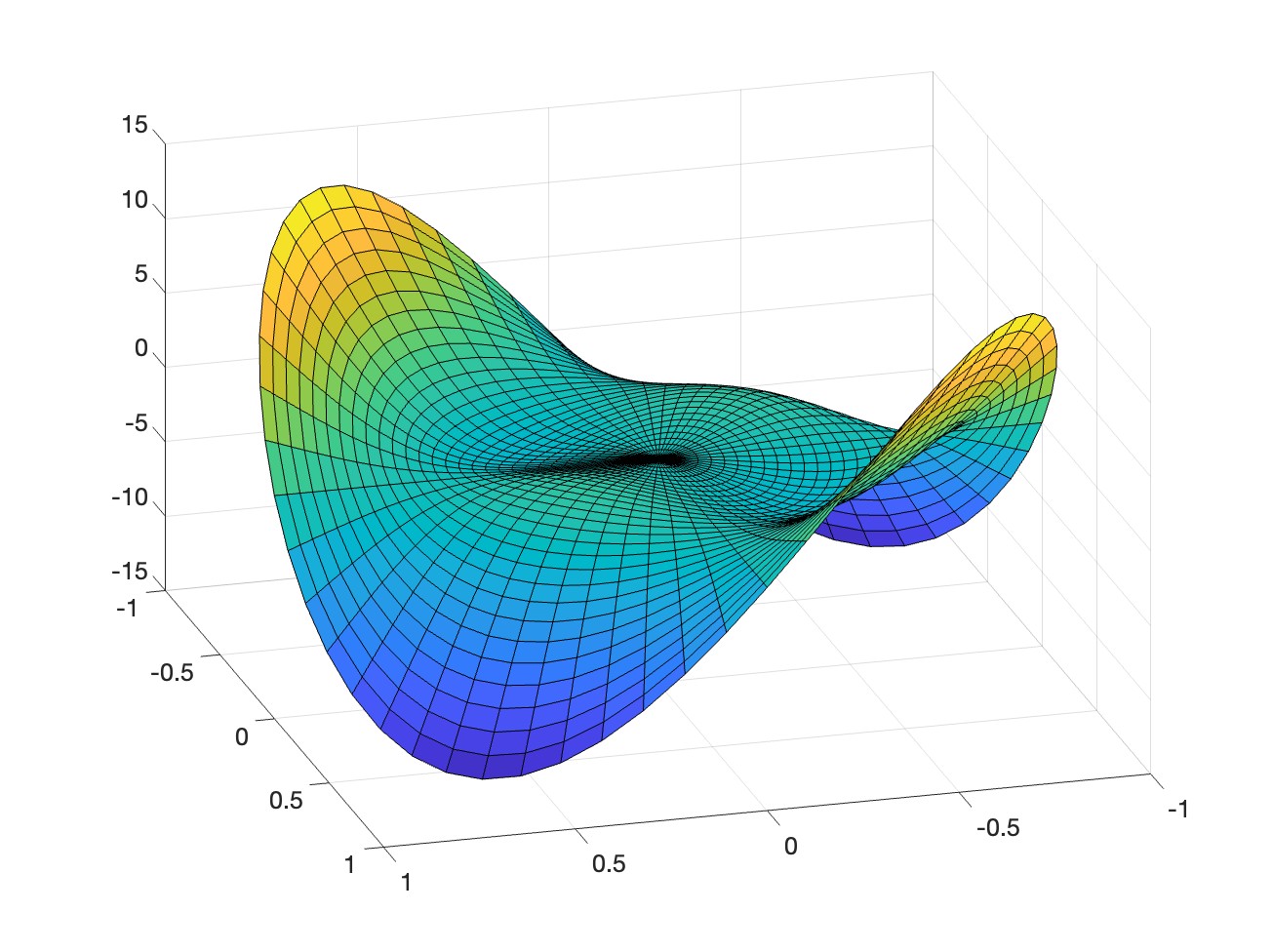}
			\end{minipage}&
			\begin{minipage}[h]{0.5\textwidth}
				\includegraphics[width=4cm]{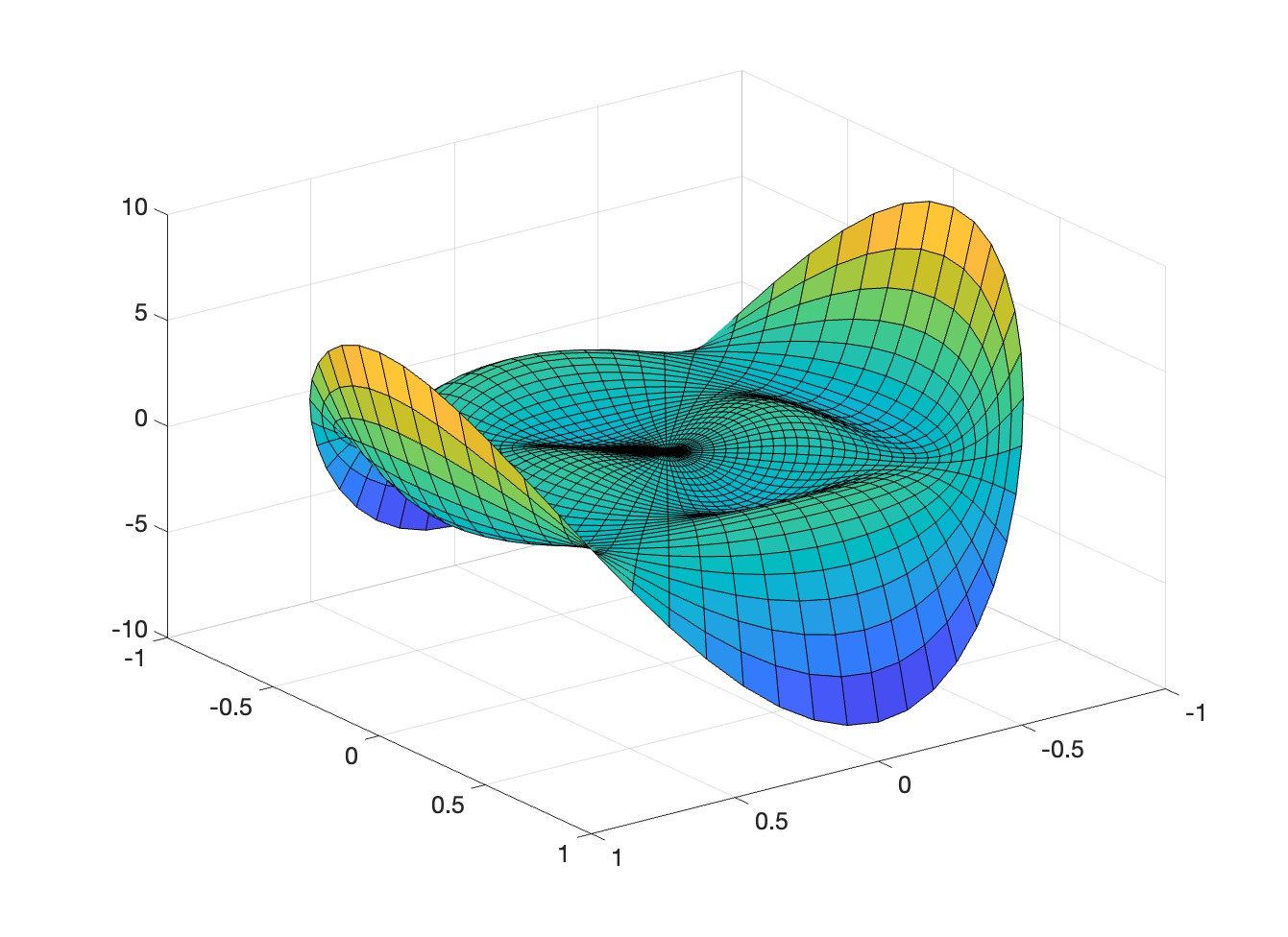}
			\end{minipage}\\
			\begin{minipage}[h]{0.45\textwidth}
				(b) Graph of $e_{12}$ part of the normalized Weighted CPSWFs $\psi_{3,2,4}^{1,1}$.
			\end{minipage} &
			\begin{minipage}[h]{0.45\textwidth}
				(d) Graph of $e_{0}$ part of the normalized Weighted CPSWFs $\psi_{5,2,(1.5)}^{1,1}$.
			\end{minipage}
		\end{tabular}
		\caption{Plots of $e_{12}$ and $e_{0}$  parts of $\psi_{n,2,\alpha}^{1,1}$.}
		\label{Weighted_CPSWFs_Figure3}
	\end{figure}
	
	\begin{figure}
		\centering
		\begin{tabular}{cc}
			\begin{minipage}[h]{0.5\textwidth}
				\includegraphics[width=4cm]{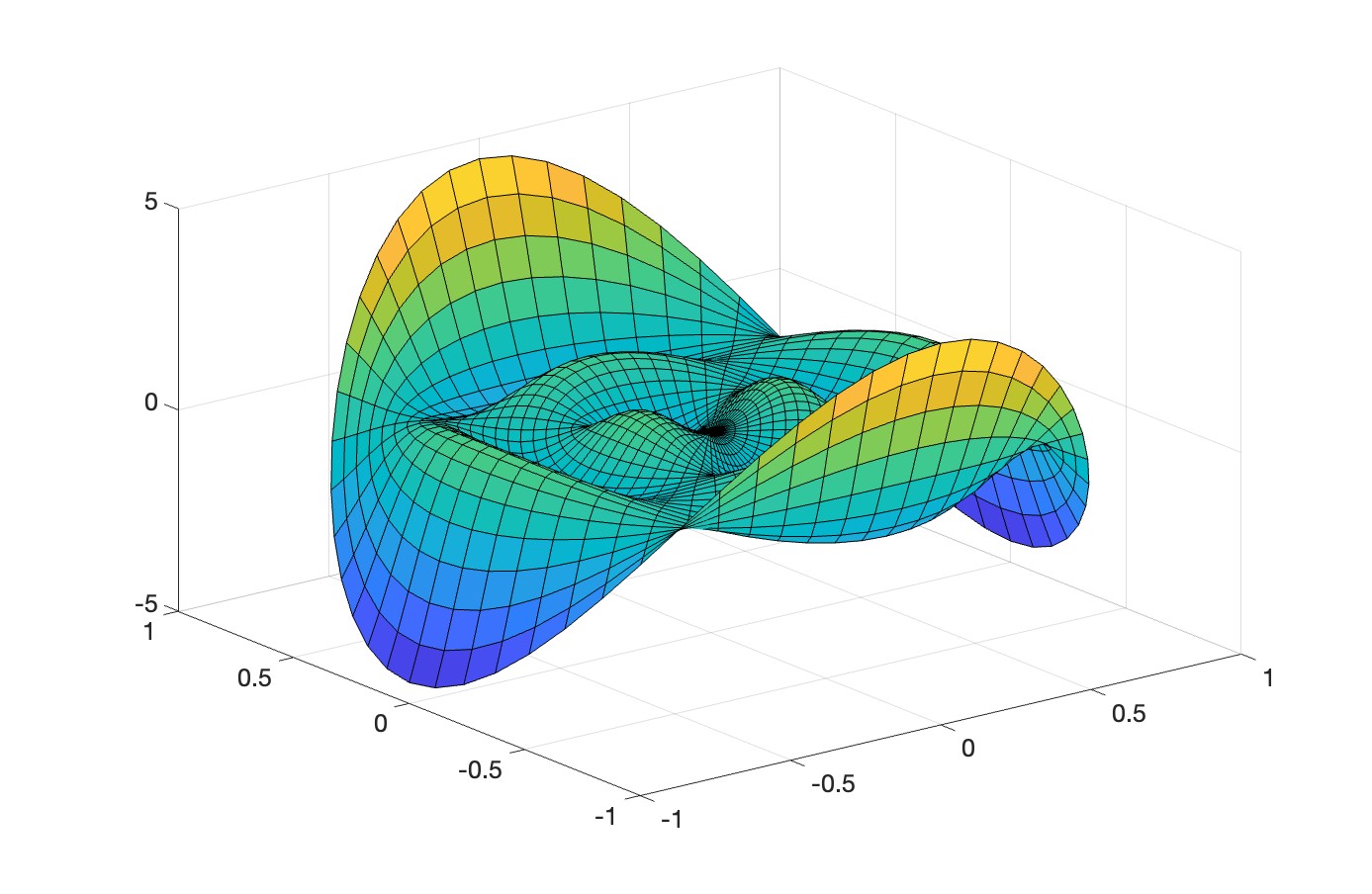}
			\end{minipage} &
			\begin{minipage}[h]{0.5\textwidth}
				\includegraphics[width=4cm]{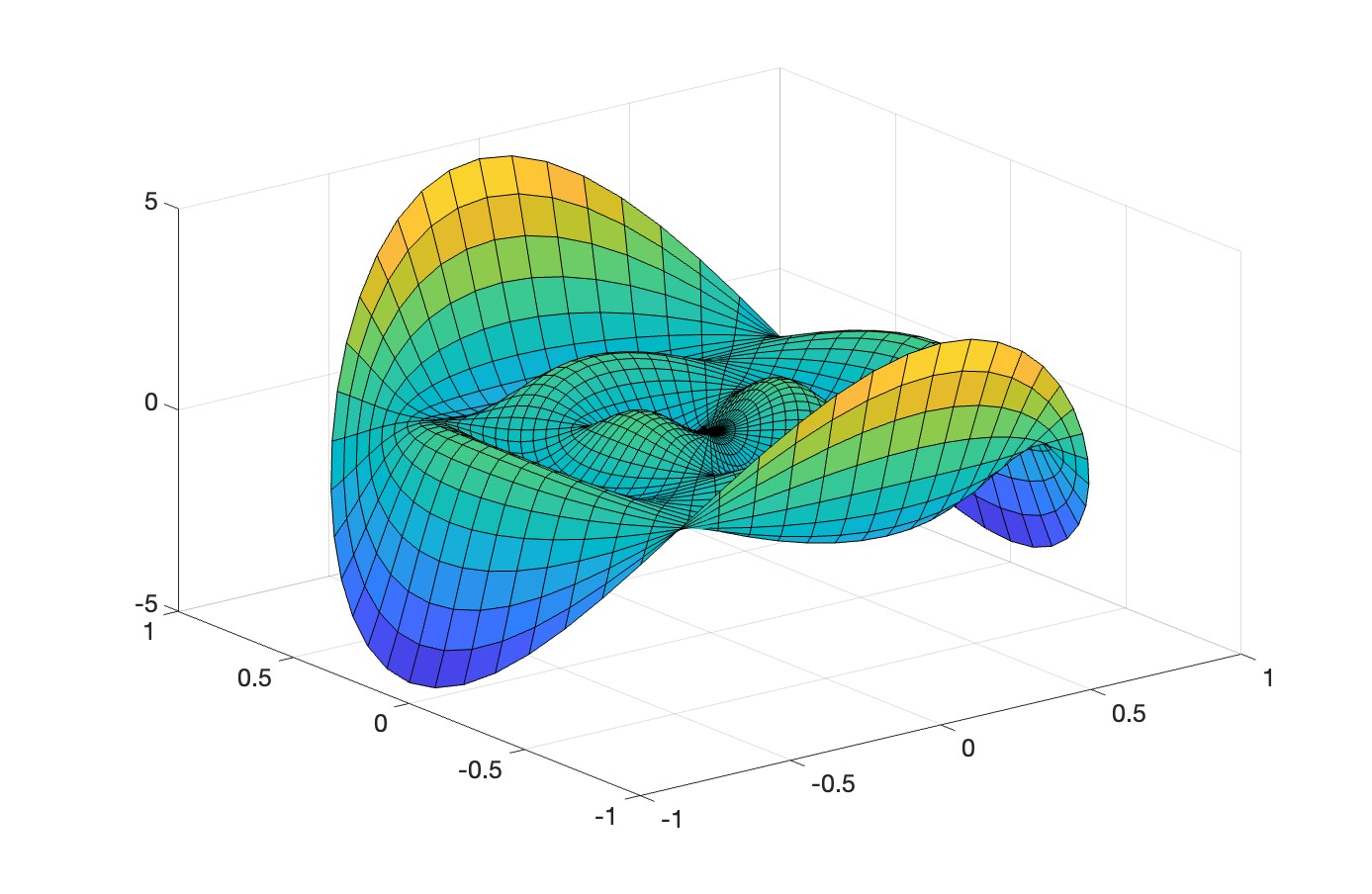}
			\end{minipage}\\
			\begin{minipage}[h]{0.45\textwidth}
				(a) Graph of $e_{1}$ part of the normalized CPSWFs $\psi_{6,2,1}^{2,2}$
			\end{minipage} & 
			\begin{minipage}[h]{0.45\textwidth}
				(b) Graph of $e_{0}$ part of the normalized CPSWFs $\psi_{7,2,1}^{1,2}$
			\end{minipage}
		\end{tabular}
		\caption{Plots of $\psi_{n,2,1}^{k,2}$.}
		\label{Weighted_CPSWFs_Figure}
	\end{figure}
	
	\section{Sturm-Liouville equation of Weighted CPSWFs}
	The main goal of this section is to prove the orthogonality of $\{ \psi_{n,m,\alpha}^{k,c,j} \}_{j,k,n}$ in $L^{2}(B(1),(1-\vert x\vert^2)^{\alpha})$ using Sturm-Liouville theory in one dimension. This orthogonality enables us to conclude that the set $\{\psi_{n,m,\alpha}^{k,c,i}:\  n\geq 0,\, k\geq 0, \, 1\leq i\leq  d_{k,m}, \alpha>-1 \} ,$ is an orthonormal basis for $L^{2}_{\alpha}({B(1)},\mathbb{C}_{m})$. 
	\begin{Remark}
		Using the way of  construction weighted CPSWFs and the fact that $C_{2N,m}^{\alpha}(Y_{k}^{j})(x)=P_{N,k,m}^{\alpha}(\vert x\vert^{2})Y_{k}^{j}(x)$ we can conclude that $\psi_{2N,m,\alpha}^{k,c,j}(x)=P_{N,m,\alpha}^{k,c}(\vert x\vert^{2})Y_{k}^{j}(x)$. In the same mannar we can write $\psi_{2N+1,m,\alpha}^{k,c,j}(x)=Q_{N,m,\alpha}^{k,c}(\vert x\vert^{2})xY_{k}^{j}(x)$.
	\end{Remark}
	Now by applying the change of variables $\vert x\vert^{2}=t$, and $\vert x\vert^{2}=\frac{s+1}{2}$ to the equation
	\begin{equation}
		\mathcal{L}_{c,\alpha}\psi_{2N,m,\alpha}^{k,c,j}(x)=\chi_{2N,m,\alpha}^{k,c}\psi_{2N,m,\alpha}^{k,c,j}(x),
	\end{equation}
	where $\mathcal{L}_{c,\alpha}$ is defined in \eqref{L_c_alpha_operator} we can obtain the Sturm-Liouville form of the differential equations corresponding to the even and odd weighted CPSWFs.
	
	\begin{Lemma}\label{even_radial_orthogonal_with_change_variable}
		The radial part of $P_{N,m,\alpha}^{k,c}(|x|^2)$ of the weighted CPSWF $\psi_{2N,m,\alpha}^{k,c,j}(x),$  satisfies
		\begin{align}
			4t(1-t)\frac{d^{2}}{dt^{2}}P_{N,m,\alpha}^{k,c}(t)&+2[m+2k-t(2+m+2k+2\alpha)]\frac{d}{dt}P_{N,m,\alpha}^{k,c}(t)\nonumber\\
			&-4\pi^2c^2tP_{N,m,\alpha}^{k,c}(t)+\chi_{2N,m,\alpha}^{k,c}P_{N,m,\alpha}^{k,c}(t)=0,\label{SL_even}
		\end{align}
		which becomes a  Sturm-Liouville differential equation after multiplying by  $g(t)=t^{k+\frac{m}{2}-1}(1-t)^{\alpha}$. Therefore, $\{P_{N,m,\alpha}^{k,c}\}_{N=0}^\infty$ may be normalised so that 
		$$\int\limits_{0}^{1}P_{N,m,\alpha}^{k,c}(t)P_{M,m,\alpha}^{k,c}(t)t^{k+\frac{m}{2}-1}(1-t)^{\alpha}\, dt=\langle P_{N,m,\alpha}^{k,c},P_{M,m,\alpha}^{k,c}\rangle_{g(t)}=\delta_{MN}.$$
		Furthermore, for each fixed integer $k\geq 0$, the collection $\{P_{N,m,\alpha}^{k,c}:\, N\geq 0\}$ is complete in the weighted space $L^2([0,1], t^{k+\frac{m}{2}-1}(1-t)^{\alpha})$ of measurable functions $f:[0,1]\to{\mathbb R}$ for which $\int\limits_0^1|f(t)|^2t^{k+\frac{m}{2}-1}(1-t)^{\alpha}\, dt <\infty$. Moreove the eigenvalues $\{ \chi_{2N,m,\alpha}^{k,c} \}_{N=0}^{\infty}$ may be distinct so that $\chi_{0,m,\alpha}^{k,c}<\chi_{2,m,\alpha}^{k,c}<\dots$.
	\end{Lemma}
	
	\begin{Remark}
		By considering $\vert x\vert^{2}=\frac{s+1}{2}$ we can see that
		\begin{align}
			(1-s^{2}) \frac{d^{2}}{ds^{2}} \tilde{P}_{N,m,\alpha}^{k,c}(s)&+ [ (k+\frac{m}{2}-1 - \alpha ) - s(k+\frac{m}{2} +\alpha + 1) ]\frac{d}{ds} \tilde{P}_{N,m,\alpha}^{k,c}(s)\nonumber\\
			&-\pi^{2}c^{2}\frac{s+1}{2}\tilde{P}_{N,m,\alpha}^{k,c}(s)=-\frac{\chi_{2N,m,\alpha}^{k,c}}{4}\tilde{P}_{N,m,\alpha}^{k,c}(s),
		\end{align}
		becomes a Sturm-Liouville differential equation after multiplying $y(s)=(s+1)^{k+\frac{m}{2}-1}(1-s)^{\alpha}$. 
	\end{Remark}
	\begin{Lemma}\label{orthogonality_of_even_and_even_weighted_CPSWFs}
		For each integer $k\geq 0$, let $\{Y_k^i\}_{i=1}^{d_{k,m}}$ be an orthonormal basis for $M_l^+(k)$. Then the even eigenfunctions of the operator
		$L_{c},$ i.e, $\{\psi_{2N,m,\alpha}^{k,c,i}:\, k\geq 0,\ 1\leq i\leq d_{k,m}\}$ are orthogonal in $L^{2}_{\alpha}(B(1),\mathbb{C}_{m})$.
	\end{Lemma}
	\begin{proof} Note that 
		\begin{align*}
			\int\limits_{B(1)}\overline{\psi_{2N,m,\alpha}^{k,c,i}(x)} &\,  \psi_{2N',m,\alpha}^{k',c,i'}(x)\, (1-\vert x\vert^{2})^{\alpha} dx
			=\int\limits_{{B(1)}}P_{N,m,\alpha}^{k,c}(\vert x\vert^{2})P_{N',m,\alpha}^{k',c}(\vert x\vert^{2})\overline{Y}_{k}^{i}(x)Y_{k'}^{i'}(x)\,  (1-\vert x\vert^{2})^{\alpha}  dx\\
			&=\int\limits_{0}^{1}P_{N,m,\alpha}^{k,c}(r^{2})P_{N',m,\alpha}^{k',c}(r^{2})r^{m+k+k'-1}\, (1-r^{2})^{\alpha} \int\limits_{S^{m-1}}\overline{Y_{k}^{i}(\omega)}Y_{k'}^{i'}(\omega )\, d\omega\, dr\\
			&=\int\limits_{0}^{1}P_{N,m,\alpha}^{k,c}(r^{2})P_{N',m,\alpha}^{k,c}(r^{2})r^{m+2k-1}\, (1-r^{2})^{\alpha} dr\delta_{kk'}\delta_{ii'}\\
			&=\frac{1}{2}\int\limits_{0}^{1}P_{N,m,\alpha}^{k,c}(t)P_{N',m,\alpha}^{k,c}(t)t^{k+\frac{m}{2}-1}\,  (1-t)^{\alpha}dt\, \delta_{kk'}\, \delta_{ii'}\\
			&=\frac{1}{2}\delta_{N N'}\, \delta_{k k'}\, \delta_{i i'}.
		\end{align*}
	\end{proof}
	
	\begin{Remark}\label{orthogonality_of_odd_and_odd_weighted_CPSWFs}
		We can prove that the radial part of the odd $\psi_{2N+1,m,\alpha}^{k,c}(x)$, i.e.,  $\{Q_{N,m,\alpha}^{k,c}\}_{N=0}^\infty$ may be orthonormal in the weight $g(t)=t^{k+\frac{m}{2}}$ so that
		$$\int\limits_{0}^{1}Q_{N,m,\alpha}^{k,c}(t)Q_{M,m,\alpha}^{k,c}(t)t^{k+\frac{m}{2}}(1-t)^{\alpha}\, dt=\langle Q_{N,m,\alpha}^{k,c},Q_{M,m,\alpha}^{k,c}\rangle_{g(t)}=\delta_{MN}.$$
		We can prove that $\{\psi_{2N+1,m,\alpha}^{k,c}(x)\}_{N,k=0}^\infty$ are orthogonal in  $L^{2}_{\alpha}(B(1),\mathbb{C}_{m})$.
	\end{Remark}
	Now using the corollary \ref{property_two_monogenic_x_between} we can prove the orthogonality of weighted CPSWFs in $L^{2}_{\alpha}(B(1),\mathbb{C}_{m})$.
	\begin{Proposition}\label{orthogonality_of_even_and_odd_CPSWFs}
		The even and odd eigenfunctions of $\mathcal{L}_{c,\alpha}$, are orthogonal, i.e.,
		$$\int\limits_{B(1)}\overline{\psi_{2N,m}^{k,c,i}(x)}\psi_{2N'+1,m}^{k',c,i'}(x)\, (1-\vert x\vert^{2})^{\alpha}dx=0,$$
		for all $N,N'\geq 0$, $k,k'\geq 0$, $1\leq i\leq d_{k,m}\,$, $1\leq i'\leq d_{k',m}$.
	\end{Proposition}
	The lemma \ref{orthogonality_of_even_and_even_weighted_CPSWFs}, remark \ref{orthogonality_of_odd_and_odd_weighted_CPSWFs}, and, proposition \ref{orthogonality_of_even_and_odd_CPSWFs} prove that $ \{ \psi_{n,m,\alpha}^{k,c,i}(x) \}_{n,k=0}^{\infty} $ are orthogonal in $L^{2}_{\alpha}(B(1),\mathbb{C}_{m})$. Now we prove they are also complete in the mentioned space, consequently, they are basis in the $L^{2}_{\alpha}(B(1),\mathbb{C}_{m})$.
	
	\begin{Th}\label{Span_Weighted_CPSWFs}
		Suppose $f$ is as in \eqref{Representation_f_all_monogenics_equation} and is supported on ${B(1)}$ with $\langle\psi_{2N,m,\alpha}^{k,c,i},f\rangle_{\alpha}=0$ for all $N\geq 0,\; k\geq 0,\; 1\leq j\leq d_{k,m},$ where $d_{k,m}$ is the dimension of $M_{l}^{+}(k).$ Then $g_{k}^{(i)}=f_{k}^{(i)}=0,$ for all $k\geq 0,\; 1\leq i\leq d_{k,m}$. 
	\end{Th}
	\begin{proof}
		Let $n=2N$, and suppose $f\in L^{2}_{\alpha}({\mathbb R}^m,{\mathbb C}_m)$ admits the expansion  \eqref{Representation_f_all_monogenics_equation} and is supported on ${B(1)}.$ By corollary \ref{property_two_monogenic_x_between} and lemma \ref{support_Lemma_f_summand}, we have
		\begin{align*}
			0=\langle\psi_{2N,m}^{k,c,i},f\rangle_{\alpha}&=\int\limits_{{B(1)}}\overline{\psi_{2N,m}^{k,c,i}(x)}f(x)\, (1-\vert x\vert^{2})^{\alpha} dx\\
			&=\int\limits_{{B(1)}}P_{N,m}^{k,c}(\vert x\vert^{2})\overline{Y_{k}^{i}(x)}\sum_{k'=0}^{\infty}\sum_{l=0}^{d_{k',m}}f_{k'}^{\ell}(\vert x\vert)Y_{k'}^{\ell}(x)\, (1-\vert x\vert^{2})^{\alpha} dx\\
			&=\sum_{k'=0}^{\infty}\sum_{\ell=1}^{d_{k',m}}\int\limits_{0}^{1}P_{N,m}^{k,c}(r^{2})f_{k'}^{\ell}(r)r^{k+k'+m-1} (1-r^{2})^{\alpha} \int\limits_{S^{m-1}}\overline{Y_{k}^{i}(\omega)}Y_{k'}^{\ell}(\omega)\, d\omega\, dr\\
			&=\int\limits_{0}^{1}P_{N,m}^{k,c}(r^{2})f_{k}^{(i)}(r)r^{2k+m-1}\, (1-\vert x\vert^{2})^{\alpha} dr\\
			&=2\int\limits_{0}^{1}P_{N,m}^{k,c}(t)f_{k}^{(i)}(\sqrt{t})t^{k+\frac{m}{2}-1}\, (1-t^{2})^{\alpha} dt.
		\end{align*}
		Hence by Lemma \ref{even_radial_orthogonal_with_change_variable} $f_{k}^{(i)}=0,$ for all $k$ and $i$. Now if we let $n=2N+1$ then we can prove that $g_{k}^{(i)}=0,$ for all $k$ and $i$ in a similar way.
	\end{proof}
	
	\section{Future Works}
	We can define $\mathcal{G}_{c,\alpha}$ from $L^{2}_{\alpha}({B}(1),\mathbb{C}_{m})$ to $L^{2}_{\alpha}({\mathbb R}^m,\mathbb{C}_{m})$ by
	\begin{equation}\label{Definition_of_gc_alpha}
		\mathcal{G}_{c,\alpha}f(x)=\chi_{{B(1)}}(x)\int\limits_{{B(1)}}e^{2\pi ic\langle x,y\rangle}f(y)\, (1-\vert y\vert^{2})^{\alpha}\, dy,
	\end{equation}
	where $ \chi_{{B(1)}} $ is the characteristic function of $B(1)$. The numerical experiments indicate that the weighted CPSWFs are eigenfunctions of the $\mathcal{G}_{c,\alpha}$ defined in \eqref{Definition_of_gc_alpha}. The analytic proofs and further properties such as approximations corresponding to the weighted CPSWFs will be reported later.

	\section*{Acknowledgment}
	\noindent The authors would like to thank the Center for Computer-Assisted Research in Mathematics and its Applications at the University of Newcastle for its continued support. JAH is supported by the Australian Research Council through Discovery Grant DP160101537. Hamed has also the Lift-off fellowship from AUSTMS. Special thanks also to Michael Speckbacher for the useful conversation in STROBL22.
	\bibliographystyle{siam}
	\bibliography{Construction_of_Weighted_Clifford_Prolate_Spheroidal_Wave_Functions}

\end{document}